\documentclass[11pt]{article}
\usepackage{amsmath,amssymb}

\def\N{\mathbb{N}}

\def\R{\mathbb{R}}

\def\C{C^{\infty}(M,\mathbb{R})}

\def\lcf{\lbrack\! \lbrack}
\def\rcf{\rbrack\! \rbrack}

\newtheorem{definition}{Definition}[section]

\newtheorem{proposition}[definition]{Proposition}
\newtheorem{theorem}[definition]{Theorem}
\newtheorem{remark}[definition]{Remark}
\newtheorem{corol}[definition]{Corollary}

\newenvironment{proof}{\noindent{\bf Proof.}}{\hfill $\blacklozenge$}

\begin{document}

\title{On twisted contact groupoids and on integration of twisted Jacobi manifolds}
\author{Fani Petalidou\\ \emph{Department of Mathematics and Statistics}
\\ \emph{University of Cyprus} \\ \emph{1678 Nicosia, Cyprus} \\ \vspace{8mm} {\small  \emph{E-mail: petalido@ucy.ac.cy}}}

\date{}
\maketitle

\begin{abstract}
We introduce the concept of twisted contact groupoids, as an extension either of contact groupoids or of twisted symplectic ones, and we discuss the integration of twisted Jacobi manifolds by twisted contact groupoids. We also investigate the very close relationships which link homogeneous twisted Poisson manifolds with twisted Jacobi manifolds and homogeneous twisted symplectic groupoids with twisted contact ones. Some examples for each structure are presented.

\vspace{3mm}
\begin{center}
\textbf{R\'esum\'e}
\end{center}

On introduit le concept des groupo\"ides de contact tordus, comme une extension \`a la fois des groupo\"ides de contact et des groupo\"ides symplectiques tordus, et on discute l' int\'egration des vari\'et\'es de Jacobi tordues par groupo\"ides de contact tordus. En plus, on \'etablit quelques relations \'etroites qui unissent les vari\'et\'es de Poisson homog\`enes tordues avec celles de Jacobi tordues et les groupo\"ides homog\`enes symplectiques tordus avec ceux de contact tordus. Des exemples pour chaque structure \'etudi\'ee sont pr\'esent\'es.
\end{abstract}

\vspace{3mm} \noindent {\bf{Keywords: }}{Twisted contact groupoid, homogeneous twisted symplectic groupoid, twisted Jacobi manifold.}

\vspace{3mm} \noindent {\bf{A.M.S. classification (2000):}} 53D10, 53D17, 58Hxx, 53C15.

\section{Introduction}
The notion of \emph{twisted Jacobi manifolds}, introduced by J.M. Nunes da Costa and the author of this present paper in \cite{jf}, is a weakened version of the notion of \emph{Jacobi mani\-folds} due to A. Lichnerowicz \cite{lch} which, at the same time, generalizes in a natural manner that of \emph{twisted Poisson manifolds} studied by P. \v{S}evera and A. Weinstein \cite{sw}. The latter appeared in Park's work on string theory \cite{p}, as well as in the work \cite{kl} on topological field theory by C. Klim\v{c}\'{i}k and T. Str\"{o}bl, and are also relevant in the theory of nonholonomic systems \cite{hgn}. In fact, the associated bracket on the space of smooth functions on the manifold does not satisfy the Jacobi identity whose failure is controlled by a generalized closed $3$-form. Keeping in mind that the Jacobi structures play a central role in the geometric prequantization of Poisson structures \cite{ch-mar-leo, zz, cr-zhu}, one of the motivations behind the study of twisted Jacobi structures is the likely role that they can play in some geometric prequantization process of twisted Poisson structures.

However, the prequantization of Poisson manifolds is also closely related to the integrability of Jacobi structures \cite{cr-zhu}. For this reason, the present paper is concerned with the search of the global geometric objects that will integrate twisted Jacobi structures.

The corresponding problem for Jacobi manifolds was treated, independently, by P. Libermann \cite{lib3}, C. Albert \cite{alb} and Y. Kerbrat and Z. Souici-Benhammadi \cite{krb}. For its approach, they introduced various concepts of \emph{contact groupoids}, as odd-dimensional counterparts of symplectic groupoids \cite{w-syml-gr}. In \cite{dz}, P. Dazord unified and clarified these concepts and he showed contact groupoids as the central tool for integrating local Lie algebras (it is well known that they are intimately connected with Jacobi structures \cite{g-lch}) and prequantizing Poisson manifolds in Weinstein's sense. While, in \cite{cr-zhu}, M. Crainic and Ch. Zhu discussed this problem using the method of $A$-paths. The analogous question for twisted Poisson manifolds was studied by A.S. Cattaneo and P. Xu \cite{cx} and, in the general context of twisted Dirac structures, by H. Bursztyn \textit{et al.} \cite{bur-al}. They presented the notion of \emph{twisted (pre)symplectic groupoids} as a natural extension of the one of symplectic groupoids \cite{w-syml-gr} and they proved that the twisted Poisson manifolds may be regarded as the infinitesimal form of twisted symplectic groupoids.

Inspired by the above studies, we introduce the concept of \emph{twisted contact groupoids}, we establish their fundamental properties (Proposition \ref{propert-propos}), and we prove that these are the global counterparts of twisted Jacobi manifolds (Theorems \ref{th-ind-jac} and \ref{integr-tj}). For this, we use the bijection which exists between twisted contact groupoids and \emph{homogeneous twisted symplectic ones} (Proposition \ref{prop-tcg-tsg}). The latter are a special case of twisted symplectic groupoids and they integrate \emph{homogeneous twisted Poisson manifolds} that are linked to twisted Jacobi manifolds with a very close relationships (Propositions \ref{poissonization} and \ref{th-tj-tp}).

The paper is organized as follows. In Section \ref{sect-tj} we review basic definitions and results concerning twisted Jacobi manifolds and homogeneous twisted Poisson ones and we give some examples of twisted contact structures. In Section \ref{sect-group}, we introduce the notion of contact groupoids, we describe their basic properties and their connection with the homogeneous twisted symplectic groupoids. Finally, we present some examples of such groupoids. In Section \ref{sect-integr}, after a brief presentation of the integrability problem of a Lie algebroid and of its solution given by M. Crainic and R.L. Fernandes \cite{cr-fer-Lie}, we prove our main theorem: \emph{An integrable twisted Jacobi manifold is integrated by a twisted contact groupoid.}

\vspace{2mm}
\noindent
\textbf{Notation:} In order to present the concepts of \emph{twisted Jacobi} and \emph{homogeneous twisted Poisson} structures on a smooth manifold $M$, we recall that, given a bivector field $\Lambda$ on $M$, the usual homomorphism of $\C$-modules $\Lambda^{\#}
: \Gamma(T^*M)\to \Gamma(TM)$, defined, for all $\zeta,\eta \in \Gamma(T^*M)$, by $\langle \eta,\Lambda^{\#}(\zeta)\rangle =
\Lambda(\zeta,\eta)$, can be extended to a homomorphism, also denoted $\Lambda^{\#}$, from $\Gamma(\bigwedge^kT^*M)$ to $\Gamma(\bigwedge^kTM)$, $k\in \N$, by setting, for any $f\in \C$, $\Lambda^{\#}(f)=f$, and, for all $\zeta \in
\Gamma(\bigwedge^k T^*M)$ and $\alpha_1,\ldots,\alpha_k \in
\Gamma(T^*M)$,
\begin{equation}\label{formule-homo}
\Lambda^{\#}(\zeta)(\alpha_1,\ldots,\alpha_k ) =
(-1)^k\zeta(\Lambda^{\#}(\alpha_1),\ldots,\Lambda^{\#}(\alpha_k
)).
\end{equation}
Moreover, for any $(\Lambda,E)\in \Gamma(\bigwedge^2(TM\times \R))$, the homomorphism of $\C$-modules
$(\Lambda,E)^{\#} : \Gamma(T^*M\times \R) \to \Gamma(TM \times \R)$ given,
for any $(\zeta,f)\in \Gamma(T^*M\times \R)$, by $(\Lambda,E)^{\#}(\zeta,f) = (\Lambda^{\#}(\zeta)+fE, -\langle \zeta, E\rangle)$, can be extended to a homomorphism $(\Lambda,E)^{\#}:\Gamma(\bigwedge^k(T^*M\times \R)) \to \Gamma(\bigwedge^k(TM\times \R))$ by setting, for all $(\zeta,\zeta') \in \Gamma
(\bigwedge^k( T^*M \times \R))$ and $(\alpha_1,f_1),
\ldots,(\alpha_k,f_k)\in \Gamma(T^*M \times \R)$,
\begin{equation*}
(\Lambda,E)^\#(\zeta,\zeta')((\alpha_1,f_1), \ldots,
(\alpha_k,f_k))
=(-1)^k(\zeta,\zeta')((\Lambda,E)^\# (\alpha_1,f_1), \ldots ,
(\Lambda,E)^\# (\alpha_k,f_k)).
\end{equation*}
We indicate \cite{im}, for any $(X_1,f_1),\ldots,(X_k,f_k) \in \Gamma(\bigwedge^k(TM\times \R))$,
\begin{equation*}
(\zeta,\zeta')((X_1,f_1), \ldots ,
(X_k,f_k)) = \zeta(X_1, \ldots,X_k) + \sum_{i=1}^k (-1)^{i+1}f_i \zeta' (X_1,\ldots, \hat{X}_i,\ldots,X_k),
\end{equation*}
where the hat denotes missing arguments. Also, following \cite{sw}, we denote by $(\Lambda^{\#}\otimes
1)(\zeta)$ the section of $(\bigwedge^{k-1}TM)\otimes T^*M$ that
acts on multivector fields by contraction with the factor in
$T^*M$. Precisely, for all $X\in \Gamma(TM)$ and $\alpha_1,\ldots,
\alpha_{k-1} \in \Gamma(T^*M)$,
\begin{equation*}\label{rel-otimes}
(\Lambda^{\#}\otimes 1)(\zeta)(\alpha_1,\ldots, \alpha_{k-1})(X) =
(-1)^k\zeta(\Lambda^{\#}(\alpha_1),\ldots,\Lambda^{\#}(\alpha_{k-1}),X).
\end{equation*}

\section{Twisted Jacobi and homogeneous twisted Poisson mani\-folds}\label{sect-tj}
A \emph{twisted Jacobi manifold} is a smooth manifold $M$ endowed with a bivector field
$\Lambda$, a vector field $E$ and a $2$-form $\omega$ such that
\begin{equation}\label{def-tj}
\frac{1}{2}[(\Lambda,E),(\Lambda,E)]^{(0,1)}=(\Lambda,E)^{\#}(d\omega,\omega).
\end{equation}
The bracket on the left hand side is the Schouten
bracket of the Lie algebroid $(TM\times \R, [\cdot,\cdot], \pi)$
over $M$ modified by the 1-cocycle $(0,1)$ of its Lie algebroid
cohomology complex with trivial coefficients \cite{im}. The pair $(d\omega,\omega)\in \Gamma(\bigwedge^3(T^*M\times \R))$ can be viewed as a closed $3$-form of $TM\times \R$ with respect to the exterior derivative operator on $\Gamma(\bigwedge(T^*M\times \R))$ defined by $([\cdot,\cdot],\pi)$ and modified by $(0,1)$ \cite{im}. Writing (\ref{def-tj}) in terms of the usual Schouten bracket, we obtain \cite{jf} its equivalent expression
\begin{equation}\label{equiv-tj}
\left\{
\begin{array}{l}
\frac{1}{2}[\Lambda,\Lambda] + E \wedge \Lambda = \Lambda^{\#}(d\omega)+
\Lambda^{\#}(\omega)\wedge E  \\
\\

[E,\Lambda] = (\Lambda^{\#} \otimes 1)(d\omega)(E)- (( \Lambda^{\#}
\otimes 1)(\omega)(E)) \wedge E.
\end{array}
\right.
\end{equation}

As explained in \cite{jf}, the space $\C$ of smooth functions on $(M,\Lambda,E,\omega)$ is equipped, just as in the case of ordinary Jacobi manifolds, with the internal composition law
\begin{equation}\label{br-j}
\{f,g\} = \Lambda(df,dg) + \langle fdg-gdf, E\rangle, \quad \quad \; f,g\in \C,
\end{equation}
that is bilinear and skew-symmetric but its Jacobi identity acquires an extra term:
\begin{equation}\label{jac-jac}
\{f,\{g,h\}\} + c.p. = (\Lambda,E)^{\#}(d\omega,\omega)((df,f),(dg,g),(dh,h)),  \quad f,g,h\in \C.
\end{equation}
So, (\ref{br-j}) is no more a Lie bracket on $\C$. However, $(\Lambda,E,\omega)$ produces a Lie algebroid
structure $(\{\cdot,\cdot\}^{\omega}, \pi \circ (\Lambda,E)^\#)$
on the vector bundle $T^*M\times \R \to M$. The
bracket on the space $\Gamma(T^*M\times \R)$ of smooth sections of $T^*M\times \R$ is
given, for all $(\zeta,f), (\eta,g)\in \Gamma(T^*M\times \R)$,
by
\begin{equation}\label{croch-form}
\{(\zeta,f),(\eta,g) \} ^{\omega} = \{(\zeta,f),(\eta,g)\} +
(d \omega,\omega)((\Lambda,E)^\# (\zeta,f), (\Lambda,E)^\#
(\eta,g), \cdot),
\end{equation}
$\{\cdot,\cdot\}$ being the Kerbrat-Souici-Benhammadi
bracket \cite{krb} on $\Gamma(T^*M\times \R)$, and the anchor map is $\pi \circ
(\Lambda,E)^\#$, where $\pi :TM\times \R \to TM$ denotes the
projection on the first factor. The section $(-E,0)$ of $TM\times \R$ is a $1$-cocycle of the Lie algebroid cohomology complex with trivial coefficients \cite{im} of $(T^*M\times \R,\{\cdot,\cdot\}^{\omega}, \pi \circ (\Lambda,E)^\#)$. The brackets (\ref{br-j}) and (\ref{croch-form}) are related by
\begin{equation*}
\{(df,f),(dg,g) \} ^{\omega} = (d\{f,g\},\{f,g\}) + (d \omega,\omega)((\Lambda,E)^\# (df,f), (\Lambda,E)^\#
(dg,g), \cdot),
\end{equation*}
whence we conclude that the mapping $f \mapsto X_f =\Lambda^\#(df)+fE = \{f,\cdot \} + \langle df,E\rangle$ from functions to their hamiltonian vectors fields is no longer a homomorphism.

\vspace{2mm}

If $E=0$, equations (\ref{equiv-tj}) are reduced to $\frac{1}{2}[\Lambda,\Lambda]=\Lambda^{\#}(d\omega)$, which means that $(\Lambda,d\omega)$ defines an exact twisted Poisson structure \cite{sw} on $M$.

\vspace{2mm}

Two characteristic examples of twisted Jacobi manifolds are the following.

\vspace{2mm}

\noindent
\emph{\textbf{Conformal twisted Jacobi manifolds (\cite{jf}):}} Let $(M,\Lambda,E,\omega)$ be a twisted Jacobi manifold and $a$ an element of $\C$ that never vanishes on $M$. We set
\begin{equation*}
\Lambda^a=a\Lambda, \quad  E^a = \Lambda^{\#}(da)+aE \quad  \mathrm{and} \quad \omega^a=\frac{1}{a}\omega.
\end{equation*}
Then, $(\Lambda^a,E^a,\omega^a)$ defines a new twisted Jacobi structure on $M$ that is called $a$-\emph{conformal} to the initially given one. Its associated bracket (\ref{br-j}) in $\C$ is given by
\begin{equation*}
\{f,g\}^a = \frac{1}{a}\{af,ag\}, \quad \quad f,g \in \C.
\end{equation*}

\vspace{2mm}
\noindent
\emph{\textbf{Twisted contact manifolds (\cite{jf-ten}):}} A \emph{twisted contact manifold} is a $2n+1$-dimensional smooth manifold $M$ equipped with an $1$-form $\vartheta$ and a $2$-form $\omega$ such that $\vartheta\wedge (d\vartheta + \omega)^n\neq 0$, everywhere in $M$. The Reeb vector field $E$ on $M$, defined by
\begin{equation}\label{reeb}
i(E)\vartheta =1 \;\;\; \mathrm{and} \;\;\; i(E)(d\vartheta
+\omega)=0,
\end{equation}
and the bivector field $\Lambda$ on $M$ whose associated morphism
$\Lambda^{\#}$ is given by
\begin{equation}\label{bivect-reeb}
\Lambda^{\#}(\vartheta)=0\quad \mathrm{and}
\quad i(\Lambda^{\#}(\zeta))(d\vartheta + \omega) = -(\zeta-
\langle \zeta,E\rangle \, \vartheta), \;\; \mathrm{for} \;\;\mathrm{all} \;\;\zeta\in\Gamma(T^*M),
\end{equation}
yield a $(d\omega,\omega)$-twisted Jacobi structure on $M$.

\vspace{2mm}
The above definition includes that $M$ is an orientable manifold and $\vartheta\wedge (d\vartheta + \omega)^n$ is a volume form on $M$. Thus, $\ker \vartheta$ and $\ker (d\vartheta + \omega)$ are complementary subbundles of $TM$, where $\ker (d\vartheta + \omega)$, called the \emph{vertical bundle}, is of rank 1 and is generated by $E$, however, $\ker \vartheta = \mathrm{Im}\Lambda^\#$, called the \emph{horizontal bundle} and denoted by $\mathcal{H}$, is of rank $2n$. If $M$ is not orientable, a twisted contact structure on $M$ cannot be defined by a single pair $(\vartheta,\omega)$ defined on the whole of $M$. It is described by an open covering $U_i$, $i\in I$, of $M$ such that each $U_i$ is endowed with a twisted contact structure $(\vartheta_i,\omega_i)$ and on the overlaps $U_i\cap U_j$ there exist nowhere vanishing functions $f_{ij}$ satisfying $\vartheta_i = f_{ij}\vartheta_j$ and $\omega_i = f_{ij}\omega_j$. Hence, the family $(U_i,\vartheta_i,\omega_i)$, $i\in I$, defines a locally conformal twisted contact structure on $M$. Of course, in this case, the above decomposition of $TM$ holds locally.

\begin{remark}\label{remark-alm-cosympl}
\emph{We recall that an \emph{almost cosymplectic structure} \cite{lib1, lch} on a $2n+1$-dimensional smooth manifold $M$ is defined by a pair $(\vartheta, \Theta)$, where $\vartheta$ is a $1$-form and $\Theta$ is a $2$-form on $M$, such that $\vartheta \wedge \Theta^n \neq 0$ everywhere on $M$. So, a twisted contact structure $(\vartheta,\omega)$ on $M$ can be viewed as an almost cosymplectic for which $\Theta = d\vartheta + \omega$. Reciprocally, each almost cosymplectic manifold $(M,\vartheta, \Theta)$ can be considered as twisted contact with $\omega = \Theta - d\vartheta$.}
\end{remark}

\noindent
\emph{\textbf{Some examples of twisted contact manifolds:}} In \cite{lib2}, P. Libermann gave a list of examples of almost cosymplectic structures: (i) on the sphere $S^5$ by embedding it in the almost complex (so, almost symplectic) manifold $S^6$ and (ii) on the quadrics homeomorphic to $S^1 \times \R^4$, $S^2 \times \R^3$ and $S^3 \times \R^2$ by embedding they in the quadrics homeomorphic to $S^2 \times \R^4$ and $S^3 \times \R^3$ which are also endowed with almost complex structures. From Remark \ref{remark-alm-cosympl}, it is clear that the almost cosymplectic structures obtained on $S^5$ and on the $5$-dimensional quadratics can be viewed as twisted contact.

\vspace{3mm}
Let $(M_1,\Lambda_1,E_1,\omega_1)$ and $(M_2,\Lambda_2,E_2,\omega_2)$ be two twisted Jacobi manifolds. A smooth map $\psi : M_1 \to M_2$ is called \emph{twisted Jacobi map} if it satisfies the following conditions
\begin{equation*}
\Lambda_2 = \psi_*\Lambda_1, \quad  E_2 = \psi_*E_1 \quad  \mathrm{and} \quad \omega_1 = \psi^*\omega_2.
\end{equation*}
Moreover, we say that $\psi : M_1 \to M_2$ is a \emph{conformal twisted Jacobi map} if there exists a non vanishing smooth function $a$ on $M_1$ such that $\psi$ is a twisted Jacobi map between $(M_1,\Lambda_1^a,E_1^a, \omega_1^a)$ and $(M_2,\Lambda_2,E_2,\omega_2)$.

\vspace{3mm}

A \emph{homogeneous twisted Poisson manifold} \cite{jf} is an exact twisted Poisson manifold $(M,\Lambda,d\omega)$ \cite{sw},
\begin{equation*}
\frac{1}{2}[\Lambda, \Lambda] = \Lambda^\#(d\omega),
\end{equation*}
endowed with a vector field $Z$, called the \emph{homothety vector field of} $M$, such that
\begin{equation}\label{def-homogeneous}
\mathcal{L}_Z\Lambda = - \Lambda \quad \mathrm{and} \quad \omega = i(Z)d\omega.\footnote{The natural definition of a \emph{homogeneous twisted Poisson manifold} would be a twisted Poisson manifold $(M,\Lambda,\varphi)$ \cite{sw} with a vector field $Z$ which would satisfy the equations $\mathcal{L}_Z\Lambda = - \Lambda$ and $\mathcal{L}_Z\varphi = \varphi$. But, $\mathcal{L}_Z\varphi = \varphi \Leftrightarrow d(i(Z)\varphi) = \varphi$, whence we get that, in the homogeneous case, $\varphi$ is exact, i.e., $\varphi = d\omega$ with $\omega = i(Z)\varphi$. Thus, $\omega = i(Z)d\omega \Leftrightarrow \omega=\mathcal{L}_Z\omega$ and $i(Z)\omega = 0$.}
\end{equation}

Now, we establish a one to one correspondence between twisted Jacobi and homogeneous twisted Poisson manifolds, as in the non-twisted framework \cite{dlm}.

\begin{proposition}\label{prop-htp-tj}
Let $(M,\Lambda,d\omega,Z)$ be a homogeneous twisted Poisson manifold and $M_0$ a $1$-codimensional submanifold of $M$ transverse to $Z$. Then, $M_0$ receives an induced twisted Jacobi structure $(\Lambda_0,E_0, \omega_0)$ characterized by one of the following properties:
\begin{enumerate}
\item
For any pair $(f,g)$ of homogeneous functions of degree $1$ with respect to $Z$, defined on an open $\mathcal{O}$ of $M$, the bracket $\{f,g\}$ is also a homogeneous function of degree 1 with respect to $Z$.
\item
Let $\varpi : U\to M_0$ be the projection, along the integral curves of $Z$, of a tubular neighborhood $U$ of $M_0$ in $M$ onto $M_0$ and $a$ a function on $U$ that is equal to $1$ on $M_0$ and homogeneous of degree $1$ with respect to $Z$. Then, the projection $\varpi$ is an $a$-conformal twisted Jacobi map.
\end{enumerate}
\end{proposition}
\begin{proof}
We have that any homogeneous function $f\in C^\infty(U,\R)$ of degree $1$ with respect to $Z$ is of type $f=a\varpi^*f_0$, where $f_0\in C^{\infty}(M_0,\R)$. Hence, $df = a\varpi^*df_0 + \varpi^*f_0da$. Let $(f,g)$ be a pair of homogeneous functions of degree $1$ with respect to $Z$, defined on an open $\mathcal{O}\subset U$ of $M$, and $(f_0,g_0)$ the corresponding pair of functions on $M_0$, i.e., $f=a\varpi^*f_0$ and $g=a\varpi^*g_0$. Then,
\begin{eqnarray}\label{eq-2}
\{f,g\} & = & \Lambda(df,dg) = \Lambda(a\varpi^*df_0 + \varpi^*f_0da, a\varpi^*dg_0 + \varpi^*g_0da) \nonumber \\
        & = & a^2\Lambda(\varpi^*df_0,\varpi^*dg_0) + a\Lambda(\varpi^*df_0, \varpi^*g_0da) + a\Lambda(\varpi^*f_0da, \varpi^*dg_0).
\end{eqnarray}
Since $\mathcal{L}_Z(a\Lambda)=0$ and $\mathcal{L}_Z(\Lambda^{\#}(da))=0$, the tensors fields $a\Lambda$ and $\Lambda^{\#}(da)$ are projectable along the integral curves of $Z$. Let $\Lambda_0 = \varpi_*(a\Lambda)$ and $E_0 = \varpi_*(\Lambda^{\#}(da))$ be their projections and $\{\cdot,\cdot\}_0$ the bracket (\ref{br-j}) defined by $(\Lambda_0,E_0)$. Thus, (\ref{eq-2}) can be written as
\begin{equation*}
\{f,g\}  =  a \varpi^*\big(\Lambda_0(df_0,dg_0) + \langle f_0dg_0 - g_0df_0, E_0\rangle \big) = a \varpi^*\{f_0,g_0\}_0,
\end{equation*}
which means that $\{f,g\}$ is also homogeneous with respect to $Z$.

We regard $(\Lambda, d\omega)$ as the twisted Jacobi structure $(\Lambda,0,\omega)$ on $M$ and we consider its $a$-conformal structure $(\Lambda^a,E^a,\omega^a)$, $\Lambda^a =a\Lambda$, $E^a = \Lambda^{\#}(da)$ and $\omega^a = \frac{1}{a}\omega$, which is also twisted Jacobi. Thus, the following equations hold.
\begin{equation}\label{syst-1}
\left\{
\begin{array}{l}
\frac{1}{2}[\Lambda^a,\Lambda^a] + E^a \wedge \Lambda^a = (\Lambda^a)^{\#}(d\omega^a)+
(\Lambda^a)^{\#}(\omega^a)\wedge E^a  \\
\\

[E^a,\Lambda^a] = ((\Lambda^a)^{\#} \otimes 1)(d\omega^a)(E^a)- (( (\Lambda^a)^{\#}
\otimes 1)(\omega^a)(E^a)) \wedge E^a.
\end{array}
\right.
\end{equation}
Since $i(Z)d\omega =\omega$ and $\mathcal{L}_Za=a$, we have $\mathcal{L}_Z\omega^a=0$ and $i(Z)\omega^a=0$, which mean that $\omega^a$ is projectable along the integral curves of $Z$. Let $\omega_0$ be the $2$-form on $M_0$ for which $\omega^a = \varpi^*\omega_0$. So, $d\omega^a = \varpi^*d\omega_0$. Also, the equation $[Z,[\Lambda^a,\Lambda^a]]=0$ means that $[\Lambda^a,\Lambda^a]$ is projectable onto $M_0$ with respect to $Z$. Its projection is the Schouten bracket of the projection of $\Lambda^a$ with itself. Hence, by projecting the system (\ref{syst-1}) along the integral curves of $Z$ and taking into account that $\Lambda_0 = \varpi_*(a\Lambda)$ and $E_0 = \varpi_*(\Lambda^{\#}(da))$, we obtain
\begin{equation*}
\left\{
\begin{array}{l}
\frac{1}{2}[\Lambda_0,\Lambda_0] + E_0 \wedge \Lambda_0 = \Lambda_0^{\#}(d\omega_0)+
\Lambda_0^{\#}(\omega_0)\wedge E_0 \\
\\

[E_0,\Lambda_0] = (\Lambda_0^{\#} \otimes 1)(d\omega_0)(E_0)- ((\Lambda_0^{\#}
\otimes 1)(\omega_0)(E_0)) \wedge E_0,
\end{array}
\right.
\end{equation*}
that signify that $(\Lambda_0,E_0,\omega_0)$ is a twisted Jacobi structure on $M_0$ and $\varpi : (M,\Lambda,0,\omega)\to (M_0,\Lambda_0,E_0,\omega_0)$ is an $a$-conformal twisted Jacobi map.
\end{proof}

\begin{proposition}[\cite{jf}]\label{poissonization}
Let $(\Lambda,E,\omega)$ be a twisted Jacobi structure on a smooth manifold $M$. Its twisted poissonization defines a homogeneous twisted Poisson structure on $\tilde{M}=M\times \R$ whose the bivector field $\tilde{\Lambda}$ and the $2$-form $\tilde{\omega}$ are given, respectively, by
\begin{equation*}
\tilde{\Lambda}=e^{-s}(\Lambda + \frac{\partial}{\partial s}\wedge E) \quad and \quad \tilde{\omega}=e^s\omega,
\end{equation*}
and the homothety vector field is $\frac{\partial}{\partial s}$, where $s$ is the canonical coordinate on $\R$. The canonical projection $\varpi : M\times \R \to M$ is a $e^s$-conformal twisted Jacobi map and the induced twisted Jacobi structure on $M$, viewed as the submanifold $M\times \{0\}$ of $\tilde{M}$, in the sense of Proposition \ref{prop-htp-tj}, is the initial given one.
\end{proposition}

\begin{corol}
When the structure $(\Lambda,E)$ on $M$ comes from a $\omega$-twisted contact form $\vartheta$ on $M$, then $\tilde{\Lambda}$ is the inverse of the twisted symplectic form $\tilde{\Omega}=d(e^s\varpi^*\vartheta)+e^s\varpi^*\omega$ on $\tilde{M}$, which is homogeneous with respect to $\frac{\partial}{\partial s}$, and reciprocally.
\end{corol}

By applying the above results, we may construct another example of twisted contact manifold.

\noindent
\textbf{\emph{The cosphere bundle of a twisted Poisson manifold.}} Let $(M,\Lambda,\varphi)$ be a twisted Poisson manifold. As D. Roytenberg has observed \cite{royt}, the cotangent bundle $T^\ast M$ of $M$ is endowed with a nondegenerate $d\omega$-twisted Poisson structure which corresponds to a $d\omega$-twisted symplectic structure $\Omega = d\vartheta + \omega$, $\vartheta$ being the Liouville form on $T^\ast M$. In a local coordinates system $(x_1,\ldots,x_n,p_1,\ldots,p_n)$ of $T^\ast M$,
\begin{equation*}
\vartheta = \sum_{i=1}^n p_i dx_i  \quad \mathrm{and} \quad \omega = \frac{1}{2}\sum_{i,j,k,l=1}^n p_i\lambda^{ij}\varphi_{jkl}dx_k\wedge dx_l,
\end{equation*}
where $\lambda^{ij}$ and $\varphi_{jkl}$ are, respectively, the local components of $\Lambda$ and $\varphi$. We remark that $(\Omega, \omega)$ is homogeneous with respect to the Liouville vector field $Z = \sum_{i=1}^{n}p_i\frac{\partial}{\partial p_i}$ of $T^\ast M$. Consider the action $\Phi$ of the multiplicative group $\R_+ = (0,+\infty)$ by dilations on the fibers of $T^\ast M \setminus \{0\}$, i.e., for any $x\in M$ and $z\in T_x^\ast M$, $\Phi(s,z)=sz$. The cosphere bundle $S^\ast M$ of $M$ is the quotient manifold $(T^\ast M \setminus \{0\})/\R_+$. By working as in \cite{rs}, we prove that the cone $S^\ast M \times \R_+$ over $S^\ast M$ equipped with the homogeneous twisted symplectic structure $(d(s\vartheta) + s\omega, s\frac{\partial}{\partial s})$, $s\in \R_+$, is twisted symplectomorphic to $(T^\ast M \setminus \{0\}, \Omega, Z)$. Thus, $S^\ast M $ receives a twisted contact structure.

\vspace{2mm}

We continue with the study of the projection of a twisted Jacobi structure along the integral curves of its vector field.

\begin{proposition}\label{th-tj-tp}
Let $(M,\Lambda,E,\omega)$ be a twisted Jacobi manifold and $M_0$ a submanifold of $M$, of codimension $1$, transverse to $E$. We denote by $\varpi : U \to M_0$ the projection onto $M_0$ of a tubular neighbourhood $U$ of $M_0$ in $M$ along the integral curves of $E$, i.e., for any $x_0 \in M_0$, $\varpi^{-1}(x_0)$ is a connected arc of the integral curve of $E$ through $x_0$, and by $\eta$ the $1$-form along $M_0$ that verifies $i(E)\eta = 1$ and $i(X)\eta= 0$, for any vector field $X$ on $M$ tangent to $M_0$, and we set $Z_0=\Lambda^{\#}(\eta)$. If $\omega = \varpi^*\omega_0$, where $\omega_0$ is a $2$-form on $M_0$, then, there exists on $M_0$ a unique exact twisted Poisson structure $(\Lambda_0,d\omega_0)$ such that $\varpi : U \to M_0$ is a twisted Jacobi map and
\begin{equation}\label{eq-Z}
\mathcal{L}_{Z_0}\Lambda_0 = - \Lambda_0 - \Lambda_0^{\#}(d\omega_0(Z_0,\cdot,\cdot)- \omega_0).
\end{equation}
Moreover, $(\Lambda_0,d\omega_0,Z_0)$ is homogeneous if and only if $\omega_0 = d\omega_0(Z_0,\cdot,\cdot)$.
\end{proposition}
\begin{proof}
Since $(\Lambda,E,\omega)$ defines a twisted Jacobi structure on $M$, (\ref{equiv-tj}) holds on $M$. Thus, because of $\omega = \varpi^*\omega_0$ and $\varpi_*E=0$,
\begin{eqnarray}\label{E-L}
[E,\Lambda] & = & (\Lambda^{\#} \otimes 1)(\varpi^*d\omega_0)(E)- (( \Lambda^{\#} \otimes 1)(\varpi^*\omega_0)(E)) \wedge E \nonumber \\
            & = & - (\varpi_*\Lambda)^{\#}(d\omega_0(\varpi_*E,\cdot,\cdot)) - (\varpi_*\Lambda)^{\#}(\omega_0(\varpi_*E,\cdot))\wedge E \,=\,0,
\end{eqnarray}
which means that $\Lambda$ is projectable along the integral curves of $E$ onto $M_0$. Let $\Lambda_0 = \varpi_*\Lambda$ be its projection. Furthermore, (\ref{E-L}) implies that $[E,[\Lambda,\Lambda]]=0$, whence we get that $[\Lambda,\Lambda]$ is also projectable onto $M_0$. Its projection $\varpi_*[\Lambda,\Lambda]$ is the Schouten bracket of the projection $\varpi_*\Lambda$ of $\Lambda$ with itself. Thus, by projecting the first equation of the system (\ref{equiv-tj}) parallel to the integral curves of $E$, we obtain
\begin{eqnarray}\label{ind-poisson}
\frac{1}{2}[\varpi_*\Lambda,\varpi_*\Lambda] + (\varpi_*E) \wedge (\varpi_*\Lambda)& = &\varpi_*(\Lambda^{\#}(\varpi^*d\omega_0))+
\varpi_*(\Lambda^{\#}(\varpi^*\omega_0))\wedge (\varpi_*E) \Leftrightarrow \nonumber \\
\frac{1}{2}[\Lambda_0,\Lambda_0] &= &\Lambda_0^{\#}(d\omega_0),
\end{eqnarray}
whence we deduce that
$(\Lambda_0, d\omega_0)$ endows $M_0$ with an exact twisted Poisson structure and that $\varpi : U \to M_0$ is a twisted Jacobi map.

Now, by integrating along the integral curves of $E$ and by restricting, if necessary, the tubular neighbourhood $U$ of $M_0$ in $M$, we can construct a function $h$ on $U$ such that $h\vert_{M_0}=0$ and $i(E)dh=1$, hence $dh\vert_{M_0}=\eta$. Let $X_h = \Lambda^{\#}(dh) + hE$ be the hamiltonian vector field of $h$ with respect to $(\Lambda,E)$. Since $[E,\Lambda](dh,\cdot) = 0$, $[E,X_h]=E$, which says that $X_h$ is projectable onto $M_0$ along the integral curves of $E$. Let $Z_0 = \varpi_*(X_h) = \varpi_*(\Lambda^{\#}(dh))$ be its projection which coincides with the restriction $X_h\vert_{M_0} = \Lambda^{\#}(\eta)$ of $X_h$ on $M_0$.

In order to establish (\ref{eq-Z}), we denote by $\{\cdot,\cdot\}$ the composition law (\ref{br-j}) on $\C$ determined by $(\Lambda,E)$ and by $\{\cdot,\cdot\}_0$ the composition law in $C^{\infty}(M_0,\R)$ defined by $\Lambda_0$. Let $(f_0,g_0)$ be a pair of smooth functions on $M_0$ and $(f,g)=(\varpi^*f_0,\varpi^*g_0)$ the associated pair of smooth functions on $M$ that are constant along the integral curves of $E$. We have
\begin{eqnarray}\label{eq-br}
\{f,g\}& = & \{\varpi^*f_0,\varpi^*g_0\} \stackrel{(\ref{br-j})}{=}\Lambda(\varpi^*df_0,\varpi^*dg_0) + \langle \varpi^*f_0\varpi^*dg_0-\varpi^*g_0\varpi^*df_0, E\rangle \nonumber \\
& = & \varpi^*\{f_0,g_0\}_0.
\end{eqnarray}
On the other hand,
\begin{eqnarray}\label{rel-1}
\lefteqn{(X_h-1)\{f,g\}  = \{h,\{f,g\}\}}  \nonumber \\
                & \stackrel{(\ref{jac-jac})}{=} & \{\{h,f\},g\} + \{f,\{h,g\}\} + (\Lambda,E)^{\#}(d\omega,\omega)((dh,h),(df,f),(dg,g)) \nonumber \\
               & = & \{(X_h-1)f,g\} + \{f, (X_h-1)g\} + (\Lambda,E)^{\#}(d\omega,\omega)((dh,h),(df,f),(dg,g)).
\end{eqnarray}
But, for any $f\in \C$ of type $f = \varpi^*f_0$ with $f_0\in C^{\infty}(M_0,\R)$,
\begin{equation}\label{rel-2}
(X_h-1)f = \varpi^*(\varpi_*(X_h-1)f_0) = \varpi^*((Z_0-1)f_0).
\end{equation}
Also,
\begin{eqnarray}\label{rel-3}
\lefteqn{(\Lambda,E)^{\#}(d\omega,\omega)((dh,h),(df,f),(dg,g))  \stackrel{(\ref{def-tj})}{=}  (\frac{1}{2}[\Lambda,\Lambda] + E\wedge \Lambda)(dh,df,dg)} \nonumber \\
& & + \,h [E,\Lambda](df,dg) + f [E,\Lambda](dg,dh) + g [E,\Lambda](dh,df) \nonumber \\
&  \stackrel{(\ref{equiv-tj})(\ref{E-L})}{=} & (\Lambda^{\#}(d\omega)+
\Lambda^{\#}(\omega)\wedge E)(dh,df,dg) \nonumber \\
&=& \Lambda^{\#}(\varpi^*d\omega_0)(dh,\varpi^*df_0,\varpi^*dg_0) + \Lambda^{\#}(\varpi^*\omega_0)(dh,\varpi^*df_0)\langle \varpi^*dg_0,E \rangle \nonumber \\
& & +\, \Lambda^{\#}(\varpi^*\omega_0)(\varpi^*df_0,\varpi^*dg_0)\langle dh,E \rangle + \Lambda^{\#}(\varpi^*\omega_0)(\varpi^*dg_0,dh)\langle \varpi^*df_0,E \rangle \nonumber \\
&\stackrel{(\ref{formule-homo})}{=} & \varpi^*(- \, d\omega_0(\varpi_*(\Lambda^{\#}(dh)),\varpi_*(\Lambda^{\#}(\varpi^*df_0)),\varpi_*(\Lambda^{\#}(\varpi^*dg_0)))) \nonumber \\
& & +\, \varpi^*(\omega_0(\varpi_*(\Lambda^{\#}(\varpi^*df_0)),\varpi_*(\Lambda^{\#}(\varpi^*dg_0)))) \nonumber \\
& = & \varpi^*(- \, d\omega_0(Z_0, \Lambda_0^{\#}(df_0),\Lambda_0^{\#}(dg_0)) + \omega_0(\Lambda_0^{\#}(df_0),\Lambda_0^{\#}(dg_0))) \nonumber \\
&\stackrel{(\ref{formule-homo})}{=} & \varpi^*((- \, \Lambda_0^{\#}(d\omega_0(Z_0,\cdot,\cdot)) + \Lambda_0^{\#}(\omega_0))(df_0,dg_0)).
\end{eqnarray}
Thus, taking into account (\ref{eq-br}), (\ref{rel-2}) and (\ref{rel-3}), (\ref{rel-1}) takes the form
\begin{eqnarray*}
\varpi^*((Z_0-1)\{f_0,g_0\}_0) & = & \{\varpi^*((Z_0-1)f_0), \varpi^*g_0\} + \{\varpi^*f_0,\varpi^*((Z_0-1)g_0)\} \nonumber \\
& & + \, \varpi^*((- \, \Lambda_0^{\#}(d\omega_0(Z_0,\cdot,\cdot)) + \Lambda_0^{\#}(\omega_0))(df_0,dg_0)) \stackrel{(\ref{eq-br})}{\Leftrightarrow} \nonumber \\
(Z_0-1)\{f_0,g_0\}_0 & = & \{(Z_0-1)f_0,g_0\}_0 + \{f_0, (Z_0-1)g_0\}_0  \\
& & + \, (- \, \Lambda_0^{\#}(d\omega_0(Z_0,\cdot,\cdot)) + \Lambda_0^{\#}(\omega_0))(df_0,dg_0) \Leftrightarrow \nonumber \\
\mathcal{L}_{Z_0}\{f_0,g_0\}_0 - \{f_0,g_0\}_0  & = & \{\mathcal{L}_{Z_0}f_0,g_0\}_0 + \{f_0, \mathcal{L}_{Z_0}g_0\}_0 - 2\{f_0,g_0\}_0 \nonumber \\
& & + \, (- \, \Lambda_0^{\#}(d\omega_0(Z_0,\cdot,\cdot)) + \Lambda_0^{\#}(\omega_0))(df_0,dg_0) \Leftrightarrow \nonumber \\
\mathcal{L}_{Z_0}\Lambda_0(df_0,dg_0) & = & -\Lambda_0(df_0,dg_0) + (- \, \Lambda_0^{\#}(\varphi_0(Z_0,\cdot,\cdot)) + \Lambda_0^{\#}(\omega_0))(df_0,dg_0),
\end{eqnarray*}
whence we deduce (\ref{eq-Z}). Also, according to (\ref{def-homogeneous}) and (\ref{eq-Z}), we get that $(\Lambda_0,d\omega_0)$ is homogeneous with respect to $Z_0$ if and only if $\omega_0 = d\omega_0(Z_0,\cdot,\cdot)$.
\end{proof}

\section{Twisted contact and homogeneous twisted symplectic groupoids}\label{sect-group}
We recall some basic results about the Lie groupoids which are needed in below and we fix our notations. For more detailed information, we suggest the following standard references \cite{mck, aw, duf-zung}.

\vspace{2mm}

Let $\Gamma \overset{\alpha}{\underset{\beta}{\rightrightarrows}}\Gamma_0$ be a \emph{Lie groupoid} with \emph{source map} $\alpha$ and \emph{target map} $\beta$. We denote by
\begin{itemize}
\item[-]
$m : \Gamma_{2} \to \Gamma$ the \emph{product map} on the set of composable pairs $\Gamma_{2} = \{(g,h)\in \Gamma\times \Gamma\, /\, \alpha(g)=\beta(h)\}$ of $\Gamma\times \Gamma$, $(g,h)\mapsto m(g,h)=g\cdot h$, which is compatible with $\alpha$ and $\beta$, i.e., for any $(g,h)\in \Gamma_2$, $\alpha(g\cdot h)=\alpha(h)$ and $\beta(g\cdot h)=\beta(g)$, and associative, i.e., for any triple $(g,h,k)$ of composable elements, $(g\cdot h)\cdot k = g \cdot (h\cdot k)$;
\item[-]
$\varepsilon: \Gamma_0 \hookrightarrow \Gamma$ the \emph{embedding} of $\Gamma_0$ into $\Gamma$, called the \emph{unit section of} $\Gamma$, which associates a unit with each element of $\Gamma_0$\footnote{For this reason, $\Gamma_0$ called, also, manifold of units.}, i.e., for any $g\in \Gamma$, $g\cdot \varepsilon(\alpha(g))=g = \varepsilon(\beta(g))\cdot g$, and satisfies $\alpha \circ \varepsilon = \beta \circ \varepsilon = \mathrm{Id}_{\Gamma_0}$;
\item[-]
$\iota: \Gamma \to \Gamma$ the \emph{inversion map}, $g\mapsto \iota(g)=g^{-1}$, which has the following properties: for any $g\in \Gamma$, $\alpha(g^{-1})=\beta(g)$, $\beta(g^{-1})=\alpha(g)$, $g^{-1}\cdot g = \varepsilon(\alpha(g))$ and $g\cdot g^{-1}=\varepsilon(\beta(g))$.
\end{itemize}

For each $g\in \Gamma$, the maps $L_g : \beta^{-1}(\alpha(g)) \to \beta^{-1}(\beta(g))$, $h\mapsto L_g(h)=g\cdot h$, and $R_g : \alpha^{-1}(\beta(g)) \to \alpha^{-1}(\alpha(g))$, $h\mapsto R_g(h)=h\cdot g$, are smooth diffeomorphisms and they are called the \emph{left translation} and the \emph{right translation} by $g$, respectively. We denote by $C_L^\infty(\Gamma, \R)$ (resp. $C_R^\infty(\Gamma,\R)$) the set of \emph{left} (resp. \emph{right}) \emph{invariant functions} on $\Gamma$, i.e., the set of smooth functions on $\Gamma$ that are invariant by the left (resp. right) translations. It is easy to verify that $C_L^\infty(\Gamma,\R) = \alpha^\ast C^\infty(\Gamma_0,\R)$ (resp. $C_R^\infty(\Gamma,\R) = \beta^\ast C^\infty(\Gamma_0,\R)$). Also, we say that a vector field $X$ on $\Gamma$ is a \emph{left invariant vector field} (resp. \emph{right invariant vector field}) if it is tangent to $\beta$-fibers, so $\beta_\ast(X)=0$, (resp. tangent to $\alpha$-fibers, so $\alpha_\ast(X) = 0$) and invariant under the left translations, i.e., $L_{g\ast}(X(h))=X(g\cdot h)$, (resp. the right translations, i.e., $R_{g\ast}(X(h))=X(h\cdot g)$).

Let $A(\Gamma) = \ker\beta_\ast\cap T_{\Gamma_0}\Gamma$ be the vector bundle over $\Gamma_0$ consisting of tangent spaces to $\beta$-fibers at the points of $\Gamma_0$. By left translations, each section of $A(\Gamma)$ gives rise to a unique left invariant vector field on $\Gamma$. Hence, the space of smooth sections $\Gamma(A(\Gamma))$ of $A(\Gamma)$ may be identified with the space of smooth left invariant vector fields on $\Gamma$ which is closed with respect the usual Lie bracket $[\cdot, \cdot]$ on $T\Gamma$. So, $\Gamma(A(\Gamma))$ inherits a Lie bracket $[\cdot, \cdot]$. $A(\Gamma)$ equipped with the above bracket and the restriction to $A(\Gamma)$ of the linear bundle map $T\alpha : T\Gamma \to T\Gamma_0$ as anchor map, becomes a Lie algebroid over $\Gamma_0$. $(A(\Gamma), [\cdot, \cdot], T\alpha)$ called the \emph{Lie algebroid of the Lie groupoid} $\Gamma \overset{\alpha}{\underset{\beta}{\rightrightarrows}}\Gamma_0$.

Recall that a \emph{Lie algebroid} \cite{mck, duf-zung} over a smooth manifold $M$ is a real vector bundle $A \to M$ together with a Lie bracket $[\cdot,\cdot]$ on the space of smooth sections $\Gamma(A)$ and a bundle map $\rho : A \to TM$, called the \emph{anchor map}, whose extension to sections satisfies the Leibniz identity
\begin{equation*}
[s_1, fs_2] = f[s_1,s_2] + (\rho(s_1)f)s_2, \quad \quad \mathrm{for}\, \mathrm{all}\; s_1, s_2 \in \Gamma(A) \;\; \mathrm{and} \;\; f\in \C.
\end{equation*}

\vspace{2mm}

Moreover, the presence of a multiplicative function $r$ on a (Lie) groupoid $\Gamma \overset{\alpha}{\underset{\beta}{\rightrightarrows}}\Gamma_0$, i.e., for any $(g,h)\in \Gamma_2$, $r(g\cdot h) = r(g) + r(h)$, permits us to define a left groupoid action of $\Gamma$ on the canonical projection $\varpi : \tilde{\Gamma}_0 = \Gamma_0\times \R \to \Gamma_0$ defined by the map
\begin{equation}\label{act-grpd}
\mathbf{ac}^{r} : \Gamma \star \tilde{\Gamma}_0 \longrightarrow \tilde{\Gamma}_0, \quad \quad \mathbf{ac}^{r}(g,(x,s))=(\beta(g), s -r(g)),
\end{equation}
where $\Gamma \star \tilde{\Gamma}_0 = \{(g,(x,s)) \in \Gamma \times \tilde{\Gamma}_0 \, / \, \alpha(g)=\varpi(x,s)=x\}$. Then, $\Gamma \star \tilde{\Gamma}_0$ itself has a (Lie) groupoid structure \cite{mck, duf-zung} with manifold of units $\tilde{\Gamma}_0$ and it is called \emph{the action groupoid associated to} $\mathbf{ac}^{r}$. It is easy to see that $\Gamma \star \tilde{\Gamma}_0$ may be identified with the product manifold $\tilde{\Gamma} = \Gamma \times \R$. Under this identification, the structural maps $\tilde{\alpha}$ (source) and $\tilde{\beta}$ (target) of $\tilde{\Gamma}$ are given, respectively, by
\begin{equation*}
\quad \quad \tilde{\alpha}(g,s)=(\alpha(g),s) \quad \quad \mathrm{and} \quad \quad \tilde{\beta}(g,s)=(\beta(g),s-r(g)), \quad \quad \mathrm{for} \; \mathrm{any} \;(g,s)\in \tilde{\Gamma}.
\end{equation*}
The set $\tilde{\Gamma}_2$ of composable pairs of $\tilde{\Gamma}\times \tilde{\Gamma}$ is identified with $\Gamma_2\times \R$ via the map
\begin{equation*}
((g_1,s_2-r(g_2)),(g_2,s_2))\mapsto ((g_1,g_2),s_2),
\end{equation*}
and the multiplication $\tilde{m}$ on $\tilde{\Gamma}_2$ is defined by
\begin{equation*}
\tilde{m}((g_1,s_2-r(g_2)),(g_2,s_2)) = (m(g_1,g_2),s_2).
\end{equation*}
While, the inversion map of $\tilde{\Gamma}$ is $\tilde{\iota}(g,s) = (\iota(g), s-r(g))$ and the unit section of $\tilde{\Gamma}$ is $\tilde{\varepsilon}(x,s) = (\varepsilon(x),s)$.

\subsection{Twisted contact groupoids}
\begin{definition}
A \emph{twisted contact groupoid} is a Lie groupoid $\Gamma \overset{\alpha}{\underset{\beta}{\rightrightarrows}}\Gamma_0$ of dimension $2n+1$ equipped with a $\omega$-twisted contact form $\vartheta$, $\omega$ being a $2$-form on $\Gamma$ of type $\omega = \alpha^{\ast}\omega_0-e^{-r}\beta^{\ast}\omega_0$, where $\omega_0\in \Gamma(\bigwedge^2T^{\ast}\Gamma_0)$ and $r$ is a smooth function on $\Gamma$, such that
\begin{equation}\label{mult-cond}
m^{\ast}\vartheta = pr_2^{\ast}(e^{-r})\cdot pr_1^{\ast}\vartheta + pr_2^{\ast}\vartheta
\end{equation}
holds on the set of composable pairs $\Gamma_2$ of $\Gamma \times \Gamma$. In (\ref{mult-cond}), $pr_i : \Gamma_2\subset \Gamma \times \Gamma \to \Gamma$ denotes the projection on the $i$-factor, $i=1,2$, of $\Gamma_2$.
\end{definition}

In the next proposition we establish the basic properties of a twisted contact groupoid.
\begin{proposition}\label{propert-propos}
Let $(\Gamma \overset{\alpha}{\underset{\beta}{\rightrightarrows}}\Gamma_0, \vartheta,\omega,r)$ be a twisted contact groupoid and $(A(\Gamma),[\cdot,\cdot],\alpha_\ast)$ its Lie algebroid. We denote by $(\Lambda_{\Gamma},E_{\Gamma},\omega)$ the $(d\omega,\omega)$-twisted Jacobi structure on $\Gamma$ defined by $(\vartheta,\omega)$. Then
\begin{enumerate}
\item[i)]
$r$ is multiplicative, i.e., for any $(g,h)\in \Gamma_2$, $r(g\cdot h)=r(g)+r(h)$, it is annulated on the units of $\Gamma$, i.e., $r\circ \varepsilon =0$, and $r\circ \iota = -r$.
\item[ii)]
$\iota^*\vartheta = -e^r\vartheta$.
\item[iii)]
$\Gamma_0$, identified with $\varepsilon(\Gamma_0)$, is a Legendrian submanifold of $\Gamma$, i.e., $\dim\Gamma_0=n$ and $\varepsilon^*\vartheta = 0$.
\item[iv)]
$E_{\Gamma}$ is a left invariant vector field and it has $r$ as first integral.
\item[v)]
For any section $X_0^l$ of $A(\Gamma)\to \Gamma_0$, we write with $X^l$ (resp. $X^r$) the left (resp. right) invariant vector field on $\Gamma$ generated by $X_0^l$ (resp. $-\iota_{\ast}X_0^l$). Let $E_0^l$ be the section of $A(\Gamma)\to \Gamma_0$ that corresponds to $E_{\Gamma}$, i.e., $E_{\Gamma} = E^l$. Then
\begin{equation}\label{eq-hamil}
\Lambda_{\Gamma}^{\#}(dr) = E^l - e^r E^r.
\end{equation}
\item[vi)]
$\iota : \Gamma \to \Gamma$ is a $-e^{-r}$-twisted conformal Jacobi map.
\item[vii)]
A vector field $Y$ on $\Gamma$ is left (resp. right) invariant if and only if there exists $(\zeta_0,f_0)\in \Gamma(T^*\Gamma_0\times \R)$ such that
\begin{equation}\label{left-right}
Y = \Lambda_{\Gamma}^{\#}(\alpha^*\zeta_0)+\alpha^*f_0E^l \quad \quad (resp. \quad Y = e^{-r}\Lambda_{\Gamma}^{\#}(\beta^*\zeta_0)+\beta^*f_0E^r).
\end{equation}
\item[viii)]
For any pair $(f_0,g_0)$ of smooth functions on $\Gamma_0$, $\{\alpha^{\ast}f_0, e^{-r}\beta^{\ast}g_0\}=0$, where $\{\cdot,\cdot\}$ denotes the bracket (\ref{br-j}) on $C^{\infty}(\Gamma, \R)$ defined by $(\Lambda_{\Gamma},E_{\Gamma})$.
\end{enumerate}
\end{proposition}
\begin{proof}
i) We consider the tangent groupoid $T\Gamma \overset{T\alpha}{\underset{T\beta}{\rightrightarrows}}T\Gamma_0$ of $\Gamma\overset{\alpha}{\underset{\beta}{\rightrightarrows}}\Gamma_0$ \cite{mrl-n}. The set of composable pairs in $T\Gamma \times T\Gamma$ is $T\Gamma_2$ with composition law $\oplus$ the tangent map $Tm$ of $m$, i.e., for any pair $(X,Y)\in T_{(g,h)}\Gamma_2$, where $X=\frac{dg(t)}{dt}\vert_{t=0}\in T_{g}\Gamma$, $Y=\frac{dh(t)}{dt}\vert_{t=0}\in T_{h}\Gamma$, and $(g(t),h(t))$, $t\in \R$, is a path in $\Gamma_2$ with $(g,h)=(g(0),h(0))$,
\begin{equation*}
X\oplus Y = T_{(g,h)}m(X,Y) = \frac{d}{dt}(g(t)\cdot h(t))\vert_{t=0}.
\end{equation*}
From the associativity of $Tm$ and (\ref{mult-cond}) we deduce the multiplicativity of $r$. Also, for any $g\in \Gamma$, we have $r(g)=r(g\cdot \varepsilon(\alpha(g))) = r(g) + r(\varepsilon(\alpha(g)))$ and $r(g)=r(\varepsilon(\beta(g))\cdot g)=r(\varepsilon(\beta(g))) + r(g)$, whence, taking into account the surjectivity of the submersions $\alpha$ and $\beta$, we get $r\circ \varepsilon = 0$. Furthermore, from the relation $g^{-1}\cdot g = \varepsilon(\alpha(g))$ and the above results, we take $r(g^{-1}\cdot g) = r(\varepsilon(\alpha(g))) \Leftrightarrow r(g^{-1})+r(g)=0 \Leftrightarrow (r\circ \iota)(g) = -r(g)$, for any $g \in \Gamma$. Thus, $r\circ \iota = -r$.

\vspace{1mm}
\noindent
ii) For any $g\in \Gamma$ and $X\in \Gamma(T\Gamma)$,
\begin{equation}\label{mult-vect}
X_g = X_g \oplus \varepsilon_\ast(\alpha_\ast X)_{\varepsilon(\alpha(g))} = \varepsilon_\ast(\beta_\ast X)_{\varepsilon(\beta(g))}\oplus X_g,
\end{equation}
since $g=g\cdot \varepsilon(\alpha(g))=\varepsilon(\beta(g))\cdot g$. So, from the multiplicativity (\ref{mult-cond}) of $\vartheta$, we get
\begin{equation*}
\vartheta_g(X_g)= \vartheta_{g\cdot \varepsilon(\alpha(g))}(X_g \oplus \varepsilon_\ast(\alpha_\ast X)_{\varepsilon(\alpha(g))})=e^{-r(\varepsilon(\alpha(g)))}\vartheta_g(X_g) + \vartheta_{\varepsilon(\alpha(g))}(\varepsilon_\ast(\alpha_\ast X)_{\varepsilon(\alpha(g))})
\end{equation*}
and
\begin{equation*}
\vartheta_g(X_g)= \vartheta_{\varepsilon(\beta(g))\cdot g}(\varepsilon_\ast(\beta_\ast X)_{\varepsilon(\beta(g))}\oplus X_g)=e^{-r(g)}\vartheta_{\varepsilon(\beta(g))}(\varepsilon_\ast(\beta_\ast X)_{\varepsilon(\beta(g))}) + \vartheta_g(X_g),
\end{equation*}
whence, taking into consideration of $r\circ \varepsilon = 0$ and of surjectivity of $T\alpha$ and $T\beta$, we obtain $\varepsilon^{\ast}\vartheta =0$. Moreover, from the properties of $\iota$ and (\ref{mult-cond}), we get
\begin{equation*}
\vartheta_{\varepsilon(\alpha(g))} =\vartheta_{g^{-1}\cdot g} = e^{-r(g)}\vartheta_{g^{-1}} + \vartheta_g \Leftrightarrow 0  =  e^{-r(g)}\vartheta_{g^{-1}} + \vartheta_g \Leftrightarrow
\vartheta_{g^{-1}} = - e^{r(g)} \vartheta_g, \quad \forall g\in \Gamma,
\end{equation*}
which means that $\iota^\ast \vartheta = -e^r\vartheta$.

\vspace{1mm}
\noindent
iii) In (ii) we proved $\varepsilon^{\ast}\vartheta =0$. So, in order to establish that $\Gamma_0\equiv \varepsilon(\Gamma_0)$ is a Legendrian submanifold of $\Gamma$, it is enough to show that $\dim \Gamma_0=n$. We consider the horizontal bundle $\mathcal{H}$ on $\Gamma$ which is of type $\mathcal{H}=\ker \vartheta$ and, for any $g\in \Gamma$, we denote by $\mathcal{H}_g$ the hyperplane through $g$. Because $\iota^\ast \vartheta = -e^r\vartheta$, $\mathcal{H}$ is invariant by inversion. Furthermore, for any $(X,Y)\in \mathcal{H}_{g}\times \mathcal{H}_{h}$ with $(g,h)\in \Gamma_2$, $m^\ast \vartheta_{(g,h)} (X,Y) \stackrel{(\ref{mult-cond})}{=}e^{-r(h)}\vartheta_{g}(X)+\vartheta_{h}(Y)\Leftrightarrow \vartheta_{g\cdot h}(X\oplus Y)=0\Leftrightarrow X\oplus Y\in \ker\vartheta_{g\cdot h}=\mathcal{H}_{g\cdot h}$, which means that $\mathcal{H}$ is closed with respect to $Tm$. Hence, by developing the argumentation of P. Dazord \cite{dz} for the pair $(\mathcal{H}, d\vartheta + \omega)$, we get $\dim \Gamma_0 =n$.

\vspace{1mm}
\noindent
iv) At every point of $\varepsilon(\Gamma_0)\equiv \Gamma_0$, the vector field $E_{\Gamma}$ can be decomposed as $E_{\Gamma}=E_0^l + \varepsilon_\ast(\beta_\ast(E_{\Gamma}))$, where $E_0^l$ is the tangent component to the $\beta$-fibre through the considered point. Let $E^l$ be the left invariant vector field on $\Gamma$ generated by $E_0^l$. We have, at each $g\in \Gamma$, $E_g^l = 0_g\oplus E^l_{\varepsilon(\alpha(g))}$. Therefore, $\vartheta_g(E^l_g)=\vartheta_{g\cdot \varepsilon(\alpha(g))}(0_g\oplus E^l_{\varepsilon(\alpha(g))})\stackrel{(\ref{mult-cond})}{=}
e^{-r(\varepsilon(\alpha(g)))}\vartheta_g(0_g)+\vartheta_{\varepsilon(\alpha(g))}(E^l_{\varepsilon(\alpha(g))})=
\vartheta_{\varepsilon(\alpha(g))}(E^l_{\varepsilon(\alpha(g))})$. But, $1=\vartheta_{\varepsilon(\alpha(g))}(E_{\Gamma \varepsilon(\alpha(g))})=\vartheta_{\varepsilon(\alpha(g))}(E_{\varepsilon(\alpha(g))}^l +
\varepsilon_\ast(\beta_\ast(E_{\Gamma}))_{\varepsilon(\alpha(g))})=\vartheta_{\varepsilon(\alpha(g))}(E_{\varepsilon(\alpha(g))}^l)$. Consequently,
\begin{equation}\label{theta-left}
\vartheta_g(E^l_g) =1.
\end{equation}
By exterior differentiation of (\ref{mult-cond}), we obtain
\begin{equation}\label{ext-dif-theta}
m^\ast d\vartheta = pr_2^\ast de^{-r}\wedge pr_1^\ast \vartheta + pr_2^\ast e^{-r}\cdot pr_1^\ast d\vartheta + pr_2^\ast d\vartheta,
\end{equation}
and we apply (\ref{ext-dif-theta}) to the pair $(E^l_g,E_{\Gamma g})$ of vectors at the point $g=g\cdot \varepsilon(\alpha(g))$. Because $r\circ \varepsilon = 0$, $E_g^l = 0_g \oplus E^l_{\varepsilon(\alpha(g))}$ and $E_{\Gamma g} = E_{\Gamma g}\oplus \varepsilon_\ast(\alpha_\ast(E_{\Gamma}))_{\varepsilon(\alpha(g))}$, we take
\begin{eqnarray}\label{theta to omega}
d\vartheta_g (E^l_g,E_{\Gamma g}) & = & de^{-r(\varepsilon(\alpha(g)))}(E^l_{\varepsilon(\alpha(g))}) \vartheta_g(E_{\Gamma g}) + d\vartheta_{\varepsilon(\alpha(g))}(E^l_{\varepsilon(\alpha(g))},\varepsilon_\ast(\alpha_\ast E_{\Gamma})_{\varepsilon(\alpha(g))})
 \stackrel{(\ref{reeb})}{\Leftrightarrow} \nonumber \\ \omega_g(E_{\Gamma g},E^l_g)  & =&  de^{-r(\varepsilon(\alpha(g)))}(E^l_{\varepsilon(\alpha(g))})  + d\vartheta_{\varepsilon(\alpha(g))}(E^l_{\varepsilon(\alpha(g))},\varepsilon_\ast(\alpha_\ast E_{\Gamma})_{\varepsilon(\alpha(g))}). \end{eqnarray}
We compare $d\vartheta_{\varepsilon(\alpha(g))}(E^l_{\varepsilon(\alpha(g))},\varepsilon_\ast(\alpha_\ast(E_{\Gamma}))_{\varepsilon(\alpha(g))})$ with $d\vartheta_g (E^l_g,E_{\Gamma g})$. From (\ref{theta-left}) and the fact $\varepsilon^\ast \vartheta = 0$, we get
\begin{equation*}\label{eqA}
d\vartheta_{\varepsilon(\alpha(g))}(E^l_{\varepsilon(\alpha(g))},\varepsilon_\ast(\alpha_\ast(E_{\Gamma}))_{\varepsilon(\alpha(g))}) = -\vartheta_{\varepsilon(\alpha(g))}([E^l_{\varepsilon(\alpha(g))},\varepsilon_\ast(\alpha_\ast(E_{\Gamma}))_{\varepsilon(\alpha(g))}])
\end{equation*}
and
\begin{equation*}\label{eqB}
d\vartheta_g (E^l_g,E_{\Gamma g})=-\vartheta_g([E^l_g,E_{\Gamma g}]).
\end{equation*}
But,
\begin{eqnarray}\label{rel []}
[E^l_g,E_{\Gamma g}] & =& [0_g \oplus E^l_{\varepsilon(\alpha(g))},E_{\Gamma g}\oplus\varepsilon_\ast(\alpha_\ast(E_{\Gamma }))_{\varepsilon(\alpha(g))}]\nonumber \\
&  = &[0_g,E_{\Gamma g}]\oplus[E^l_{\varepsilon(\alpha(g))},\varepsilon_\ast(\alpha_\ast(E_{\Gamma}))_{\varepsilon(\alpha(g))}].
\end{eqnarray}
So, $\vartheta_g([E^l_g,E_{\Gamma g}])\stackrel{(\ref{mult-cond})(\ref{rel []})}{=}\vartheta_{\varepsilon(\alpha(g))}([E^l_{\varepsilon(\alpha(g))},\varepsilon_\ast(\alpha_\ast(E_{\Gamma }))_{\varepsilon(\alpha(g))}])$, which yields
\begin{equation}\label{eqC}
d\vartheta_{\varepsilon(\alpha(g))}(E^l_{\varepsilon(\alpha(g))},\varepsilon_\ast(\alpha_\ast(E_{\Gamma}))_{\varepsilon(\alpha(g))}) = d\vartheta_g (E^l_g,E_{\Gamma g}).
\end{equation}
Consequently, (\ref{theta to omega}) can be written as
\begin{equation}\label{eqD}
(d\vartheta_g+\omega_g)(E_{\Gamma g},E^l_g) = de^{-r(\varepsilon(\alpha(g)))}(E^l_{\varepsilon(\alpha(g))})\cdot \vartheta_g(E_{\Gamma g}),
\end{equation}
whence, taking into account (\ref{reeb}), we deduce that $de^{-r(\varepsilon(\alpha(g)))}(E^l_{\varepsilon(\alpha(g))})=0$. Hence, from the last equation and (\ref{eqD}), we take
\begin{equation*}
(d\vartheta_g+\omega_g)(E^l_g, \cdot) = - de^{-r(\varepsilon(\alpha(g)))}(E^l_{\varepsilon(\alpha(g))})\cdot \vartheta_g =0,
\end{equation*}
which means that $E_g^l$ and $E_{\Gamma g}$ are parallel vectors of $T_g\Gamma$. Also, we have $\vartheta_g(E_{\Gamma g})\stackrel{(\ref{reeb})}{=}1 \stackrel{(\ref{theta-left})}{=}\vartheta_g(E^l_g)$. Hence, we conclude that, for any $g\in \Gamma$, $E_{\Gamma g} = E^l_g$, which means that $E_{\Gamma}$ is left invariant.

Next, we will show that $r$ is a first integral of $E_{\Gamma}$. At every point of $\varepsilon(\Gamma_0)\equiv \Gamma_0$, the vector field $E_{\Gamma}$ can be decomposed as $E_{\Gamma}=E_0^r + \varepsilon_\ast(\alpha_\ast(E_{\Gamma}))$, where $E_0^r$ is the tangent component to the $\alpha$-fibre through the considered point. Let $E^r$ be the right invariant vector field on $\Gamma$ generated by $E_0^r$. Because, for any $g \in \Gamma$, $g=\varepsilon(\beta(g))\cdot g$ and $E^r_g = E^r_{\varepsilon(\beta(g))}\oplus 0_g$, $\vartheta_g(E^r_g)=\vartheta_{\varepsilon(\beta(g))\cdot g}(E^r_{\varepsilon(\beta(g))}\oplus 0_g)\stackrel{(\ref{mult-cond})}{=}e^{-r(g)}\vartheta_{\varepsilon(\beta(g))}(E^r_{\varepsilon(\beta(g))})$. But, $1\stackrel{(\ref{reeb})}{=}\vartheta_{\varepsilon(\beta(g))}(E_{\Gamma \varepsilon(\beta(g))})=\vartheta_{\varepsilon(\beta(g))}(E^r_{\varepsilon(\beta(g))})$. So, $\vartheta_g(E^r_g)=e^{-r(g)}$. Also,
\begin{eqnarray*}
d\vartheta_g(E^r_g,E_{\Gamma g})  =  d\vartheta_{\varepsilon(\beta(g))\cdot g}(E^r_{\varepsilon(\beta(g))}\oplus 0_g, 0_{\varepsilon(\beta(g))}\oplus E_{\Gamma g})   & \stackrel{(\ref{reeb})(\ref{ext-dif-theta})}{\Leftrightarrow}& \\
\omega_g(E_{\Gamma g},E^r_g)  = -de^{-r(g)}(E_{\Gamma g})\cdot \vartheta_{\varepsilon(\beta(g))}(E^r_{\varepsilon(\beta(g))})  & \Leftrightarrow & \\
\omega_{0 \alpha(g)}(\alpha_{\ast}(E_{\Gamma})_{\alpha(g)},\alpha_{\ast}(E^r)_{\alpha(g)}) - \omega_{0 \beta(g)}(\beta_{\ast}(E_{\Gamma})_{\beta(g)},\beta_\ast(E^r)_{\beta(g)})   =  -de^{-r(g)}(E_{\Gamma g})   & \Leftrightarrow & \\
0 =-de^{-r(g)}(E_{\Gamma g}),
\end{eqnarray*}
whence we conclude $\mathcal{L}_{E_{\Gamma}}r = 0$.

\vspace{1mm}
\noindent
v) We consider the hamiltonian vector field $X_{e^{-r}}=\Lambda_{\Gamma}^{\#}(de^{-r})+e^{-r}E_{\Gamma}$ of $e^{-r}$ and the right invariant vector field $E^r$ on $\Gamma$ defined in the proof of (v). We apply (\ref{ext-dif-theta}) to the pair $(E^r,X_{e^{-r}})$ at a point $g=\varepsilon(\beta(g))\cdot g$ of $\Gamma$ and we prove, as in (iv), that $X_{e^{-r}} = E^r$. Since $E_{\Gamma}=E^l$ at the points of $\varepsilon(\Gamma_0)\equiv \Gamma_0$, we have $E_0^l = E_0^r + \varepsilon_{\ast}(\alpha_{\ast}(E_0^l))$. Thus, $0= \beta_\ast(E_0^r + \varepsilon_{\ast}(\alpha_{\ast}(E_0^l))) = \beta_{\ast}(E_0^r) + \alpha_{\ast}(E_0^l) = \beta_{\ast}(E_0^r) + \beta_{\ast}(\iota_{\ast}E_0^l)$. Because $E_0^r + \iota_{\ast}E_0^l$ is tangent to $\alpha$-fibers, the above equation yields that $E_0^r + \iota_{\ast}E_0^l = 0$, whence we deduce that $E^r$ coincide with the right invariant vector field on $\Gamma$ generated by $-\iota_{\ast}E_0^l$. Hence, (\ref{eq-hamil}) holds.

\vspace{1mm}
\noindent
vi) It is enough to prove that $\iota_{\ast}(-e^{-r}\Lambda_{\Gamma})=\Lambda_{\Gamma}$, $\iota_{\ast}X_{-e^{-r}}=E_{\Gamma}$ and $\iota^\ast\omega = -e^r\omega$. Since $\alpha \circ \iota = \beta$, $\beta \circ \iota = \alpha$ and $r\circ \iota = -r$, we have $\iota^\ast \omega = \iota^\ast(\alpha^\ast \omega_0 - e^{-r}\beta^\ast\omega_0)= \beta^\ast \omega_0 - e^{-r\circ \iota}\alpha^\ast \omega_0 = - e^r\omega$. Also, $\iota_{\ast}X_{-e^{-r}} = \iota_{\ast}(-E^r) = E^l$ and, for any $\zeta \in T^\ast _{g^{-1}}\Gamma$,
$$
\begin{array}{l}
\iota^\ast [i\big((T_g\iota  (-e^{-r(g)}\Lambda_{\Gamma g})^\# \,^tT_g\iota)(\zeta) \big)(d\vartheta +\omega)_{g^{-1}}] =   i\big((-e^{-r(g)}\Lambda_{\Gamma g})^\#(\iota^\ast \zeta)\big)\iota^\ast[(d\vartheta +\omega)_{g^{-1}}]  \nonumber \\
= i\big((-e^{-r(g)}\Lambda_{\Gamma g})^\#(\iota^\ast \zeta)\big) [-e^{r(g)}(d\vartheta + \omega)_g -e^{r(g)}dr\wedge \vartheta_g]    \nonumber \\
  = i\big( \Lambda_{\Gamma g}^\#(\iota^\ast \zeta)\big)(d\vartheta + \omega)_g + \Lambda_{\Gamma g}\big(\iota^\ast\zeta,dr \big)\vartheta_g   \nonumber \\
  \stackrel{(\ref{bivect-reeb})(\ref{eq-hamil})}{=}  - \iota^\ast \zeta + \langle \iota^\ast \zeta, E_{\Gamma g}\rangle \vartheta_g - \langle \iota^\ast \zeta, E_g^l\rangle \vartheta_g  + \langle \iota^\ast \zeta, e^{r(g)}E^r_g\rangle \vartheta_g  \nonumber \\
 \stackrel{(ii)}{=}  - \iota^\ast \zeta - \langle \zeta, \iota_\ast E^r_g\rangle \iota^\ast(\vartheta_{g^{-1}})
 =  \iota^\ast \big[-\zeta + \langle \zeta, E_{\Gamma g^{-1}}\rangle \vartheta_{g^{-1}}\big] = \iota^\ast [i(\Lambda_{\Gamma g^{-1}}^\#(\zeta))(d\vartheta + \omega)_{g^{-1}}],
\end{array}
$$
whence we obtain $T\iota \circ (-e^{-r}\Lambda_{\Gamma})^\#\circ\, ^tT\iota  = \Lambda_{\Gamma}^\#$, since $\mathcal{H}=\mathrm{Im}\Lambda_{\Gamma}^{\#}$ is invariant by $\iota$.

\vspace{1mm}
\noindent
vii) We establish (\ref{left-right}) as in Proposition 4.3 of \cite{dz}. For the proof we use the pair $(\mathcal{H},d\vartheta + \omega)$ and the fact that $\omega = \alpha^\ast \omega_0 - e^{-r}\beta^\ast \omega_0$.

\vspace{1mm}
\noindent
viii) We consider a pair $(f_0,g_0)$ of smooth functions on $\Gamma_0$ and we compute the bracket $\{\alpha^\ast f_0, e^{-r}\beta^\ast g_0\}$. Since $T\beta \circ \Lambda_{\Gamma}^{\#}\circ \,^tT\alpha =0$ (see \cite{dz}), $\mathcal{L}_{E_{\Gamma}}r = 0$ and $\beta_\ast(E_{\Gamma})=0$, we have
\begin{equation*}
\{\alpha^\ast f_0, e^{-r}\beta^\ast g_0\} = e^{-r}\beta^\ast g_0 \,\langle \alpha^\ast df_0, \Lambda_{\Gamma}^{\#}(dr)-E_{\Gamma}\rangle.
\end{equation*}
But, in (v) we proved that $\Lambda_{\Gamma}^{\#}(dr)-E_{\Gamma}=-e^rX_{e^{-r}}$ and $X_{e^{-r}}$ is right invariant. So, $\alpha_\ast (\Lambda_{\Gamma}^{\#}(dr)-E_{\Gamma})=0$ and $\{\alpha^\ast f_0, e^{-r}\beta^\ast g_0\} =0$. The last equation means that the twisted contact hamiltonian vector fields corresponding to the left invariant functions of $\Gamma$ have as first integrals the right invariant functions of $\Gamma$ multiplied by $e^{-r}$, and reciprocally.
\end{proof}

\vspace{2mm}

Now, we can formulate the following theorem.
\begin{theorem}\label{th-ind-jac}
Let  $(\Gamma \overset{\alpha}{\underset{\beta}{\rightrightarrows}}\Gamma_0,\vartheta,\omega,r)$, $\omega = \alpha^{\ast}\omega_0-e^{-r}\beta^{\ast}\omega_0$ with $\omega_0\in \Gamma(\bigwedge^2T^{\ast}\Gamma_0)$, be a twisted contact groupoid, $(A(\Gamma),[\cdot , \cdot], T\alpha)$ its associated Lie algebroid, and $(\Lambda_{\Gamma},E_{\Gamma},\omega)$ the twisted Jacobi structure on $\Gamma$ defined by the $\omega$-twisted contact form $\vartheta$. Then
\begin{enumerate}
\item[1)]
$(\Lambda_{\Gamma},E_{\Gamma},\omega)$ induces on $\Gamma_0$ a twisted Jacobi structure $(\Lambda_0,E_0,\omega_0)$ such that $\alpha : \Gamma\to \Gamma_0$ is a twisted Jacobi map and $\beta: \Gamma \to \Gamma_0$ is a $-e^{-r}$-conformal twisted Jacobi map.
\item[2)]
The map
\begin{equation*}\label{isom}
\begin{tabular}{rcl}
$\mathfrak{I} : \Gamma(T^*\Gamma_0\times \R)$ & $\longrightarrow$ & $\Gamma(A(\Gamma))$ \\
$(\zeta_0,f_0)$ & $\mapsto$ & $\Lambda_{\Gamma}^{\#}(\alpha^*\zeta_0)+\alpha^*f_0E^l$
\end{tabular}
\end{equation*}
induces a Lie algebroid isomorphism between $(T^*\Gamma_0\times \R,\{\cdot,\cdot\}^{\omega_0}, \pi \circ (\Lambda_0,E_0)^\#)$, which is the Lie algebroid over $\Gamma_0$ determined by $(\Lambda_0,E_0,\omega_0)$, and $(A(\Gamma),[\cdot , \cdot], T\alpha)$.
\end{enumerate}
\end{theorem}
\begin{proof}
1) Let $(f_0,g_0)$ be a pair of smooth functions on $\Gamma_0$ and $(\alpha^\ast f_0,\alpha^\ast g_0)$ the corresponding pair of left invariant functions on $\Gamma$. We have
\begin{equation*}
\{\alpha^\ast f_0,\alpha^\ast g_0\} = \Lambda_{\Gamma}(\alpha^\ast df_0, \alpha^\ast dg_0)+ \langle \alpha^\ast f_0 \alpha^\ast dg_0 - \alpha^\ast g_0 \alpha^\ast df_0, E_{\Gamma}\rangle.
\end{equation*}
Obviously, the function $\Lambda_{\Gamma}(\alpha^\ast df_0, \alpha^\ast dg_0)$ is left invariant. On the other hand, since $E_{\Gamma}$ is left invariant, it acts by differentiation on $C^\infty_L(\Gamma,\R)=\alpha^\ast C^\infty(\Gamma_0,\R)$ \cite{aw}. Thus, $\langle \alpha^\ast f_0 \alpha^\ast dg_0 - \alpha^\ast g_0 \alpha^\ast df_0, E_{\Gamma}\rangle$ is also a left invariant function on $\Gamma$. So, $\{\alpha^\ast f_0,\alpha^\ast g_0\}$ is left invariant. Consequently, there exists $h_0\in C^\infty(\Gamma_0,\R)$, $h_0=\varepsilon^\ast \{\alpha^\ast f_0,\alpha^\ast g_0\}$, such that $\{\alpha^\ast f_0,\alpha^\ast g_0\} = \alpha^\ast h_0$. The map $(f_0,g_0)\mapsto h_0$ endows $C^\infty(\Gamma_0,\R)$ with an internal composition law $\{\cdot,\cdot\}_0$ which is the bracket (\ref{br-j}) defined by a twisted Jacobi structure $(\Lambda_0,E_0,\omega_0)$ on $\Gamma_0$. Effectively, for any $u_0,v_0\in C^\infty(\Gamma_0)$,
\begin{eqnarray*}\label{eqF}
\{\alpha^\ast u_0, \alpha^\ast v_0\} &\stackrel{(\ref{br-j})}{=}& \Lambda_{\Gamma}(\alpha^\ast du_0,\alpha^\ast dv_0) + \langle \alpha^\ast u_0 \alpha^\ast dv_0 - \alpha^\ast v_0 \alpha^\ast du_0, E_{\Gamma}\rangle \Leftrightarrow \nonumber \\
\alpha^\ast \{u_0,v_0\}_0 & = & \alpha^\ast \big(\Lambda_0(du_0,dv_0)+\langle u_0 dv_0 - v_0du_0, E_0\rangle\big),
\end{eqnarray*}
where $\Lambda_0$ is the bivector field on $\Gamma_0$ associated to $\Lambda_0^\# = T\alpha \circ \Lambda^\#_{\Gamma}\circ \,^tT\alpha$ and $E_0 = \alpha_\ast(E_{\Gamma})$. Since $\alpha$ is a submersion surjective, the last equation means that $\{\cdot,\cdot\}_0$ is the bracket (\ref{br-j}) defined by $(\Lambda_0,E_0)$. Moreover, for any $u_0,v_0,w_0 \in C^\infty(\Gamma_0)$,
\begin{eqnarray}\label{eqE}
\lefteqn{\{\alpha^\ast u_0, \{\alpha^\ast v_0, \alpha^\ast w_0\}\} + c.p  \stackrel{(\ref{jac-jac})(\ref{equiv-tj})}{=}  \big(\Lambda^\#_{\Gamma}(d\omega)+\Lambda^\#_{\Gamma}(\omega)\wedge E_{\Gamma}\big)(\alpha^\ast du_0, \alpha^\ast dv_0, \alpha^\ast dw_0)} \nonumber \\
& & + \, \alpha^\ast u_0 \big((\Lambda^\#_{\Gamma}\otimes 1)(d\omega)(E_{\Gamma})-(\Lambda^\#_{\Gamma}\otimes 1)(\omega)(E_{\Gamma})\wedge E_{\Gamma}\big)(\alpha^\ast dv_0, \alpha^\ast dw_0) + c.p.
\end{eqnarray}
The left side of (\ref{eqE}) is equal to $\alpha^\ast\big(\{u_0,\{v_0,w_0\}_0\}_0 + c.p \big)$. However, taking into consideration that $\omega = \alpha^\ast \omega_0 - e^{-r}\beta^\ast \omega_0$, $\beta_\ast E_{\Gamma}=0$ and $\mathcal{L}_{E_{\Gamma}}r=0$, the right side of (\ref{eqE}) can be written as
\begin{eqnarray*}
\lefteqn{\alpha^\ast \big((\Lambda^\#_0(d\omega_0)+\Lambda^\#_0(\omega_0)\wedge E_0)(du_0,dv_0,dw_0)\big)}\\
& & +\, \big[\alpha^\ast u_0 \alpha^\ast\big(((\Lambda^\#_0\otimes 1)(d\omega_0)(E_0)-(\Lambda^\#_0\otimes 1)(\omega_0)(E_0)\wedge E_0)(dv_0,dw_0)\big) + c.p.\big].
\end{eqnarray*}
From the above results and the surjectivity of the submersion $\alpha$, we deduce that, for any $u_0,v_0,w_0 \in C^\infty(\Gamma_0,\R)$,
\begin{eqnarray*}
\lefteqn{\{u_0,\{v_0,w_0\}_0\}_0 + c.p = (\Lambda^\#_0(d\omega_0)+\Lambda^\#_0(\omega_0)\wedge E_0)(du_0,dv_0,dw_0)} \nonumber \\
& & +\, u_0\big(((\Lambda^\#_0\otimes 1)(d\omega_0)(E_0)-(\Lambda^\#_0\otimes 1)(\omega_0)(E_0)\wedge E_0)(dv_0,dw_0)\big) + c.p.,
\end{eqnarray*}
which means that $(\Lambda_0,E_0,\omega_0)$ defines a twisted Jacobi structure on $\Gamma_0$. By construction, $\alpha$ is a twisted Jacobi map.

Since $\alpha$ is a twisted Jacobi map and $\iota$ is a $-e^{-r}$-conformal twisted Jacobi map (see Proposition \ref{propert-propos}), their composition $\beta = \alpha \circ \iota$ is a $-e^{-r}$-conformal twisted Jacobi map.

\vspace{1mm}
\noindent
2) In order to show that $\mathfrak{I} : \Gamma(T^*\Gamma_0\times \R)\longrightarrow \Gamma(A(\Gamma))$ induces a Lie algebroid isomorphism, it suffices to prove that $\mathfrak{I}$ is compatible with the brackets and the anchors maps and that $\ker \mathfrak{I} = (0,0)$. We denote also by $\alpha^\ast$ the map from $\Gamma(T^\ast\Gamma_0\times \R)$ to $\Gamma(T^\ast \Gamma \times \R)$ given, for any $(\zeta_0,f_0)\in \Gamma(T^\ast\Gamma_0\times \R)$, by $\mathbf{\alpha^\ast}(\zeta_0,f_0) = (\alpha^\ast\zeta_0, \alpha^\ast f_0)$. Being $E_{\Gamma}=E^l$, we have $\mathfrak{I}= \pi \circ (\Lambda_{\Gamma},E_{\Gamma})^\#\circ \,^tT\alpha$. Since $\alpha$ is a twisted Jacobi map, $T\beta \circ \Lambda_{\Gamma}^\# \circ\,^tT\alpha = 0$ and $E_{\Gamma}$ is left invariant, it is easy to check that $\alpha^\ast$ is a Lie algebras homomorphism, i.e., for all $(\zeta_0,f_0), (\eta_0,g_0)\in \Gamma(T^\ast\Gamma_0\times \R)$,
\begin{equation}\label{alfa-homo}
\alpha^\ast(\{(\zeta_0,f_0), (\eta_0,g_0)\}^{\omega_0}) = \{\alpha^\ast(\zeta_0,f_0),\alpha^\ast(\eta_0,g_0)\}^\omega.
\end{equation}
Thus, by applying the anchor map $\pi \circ (\Lambda_{\Gamma},E_{\Gamma})^\#$ to (\ref{alfa-homo}), we obtain
\begin{eqnarray*}
\mathfrak{I}(\{(\zeta_0,f_0), (\eta_0,g_0)\}^{\omega_0})
& = & (\pi \circ (\Lambda_{\Gamma},E_{\Gamma})^\#)\{\alpha^\ast(\zeta_0,f_0),\alpha^\ast(\eta_0,g_0)\}^\omega \nonumber \\
& = & [(\pi \circ (\Lambda_{\Gamma},E_{\Gamma})^\#)(\alpha^\ast(\zeta_0,f_0)),(\pi \circ (\Lambda_{\Gamma},E_{\Gamma})^\#)(\alpha^\ast(\eta_0,g_0))] \nonumber \\
& = & [\mathfrak{I}((\zeta_0,f_0)),\mathfrak{I}((\eta_0,g_0))],
\end{eqnarray*}
which signifies that $\mathfrak{I}$ is compatible with the brackets on $\Gamma(T^*\Gamma_0\times \R)$ and $\Gamma(A(\Gamma))$. As well, we have $T\alpha \circ \mathfrak{I} =T\alpha \circ \pi \circ (\Lambda_{\Gamma},E_{\Gamma})^\#\circ \,^tT\alpha = \pi \circ T\alpha \circ (\Lambda_{\Gamma},E_{\Gamma})^\#\circ \,^tT\alpha =\pi\circ (\Lambda_0,E_0)^\#$, whence we get the compatibility of $\mathfrak{I}$ with the anchors of $T^*\Gamma_0\times \R$ and $A(\Gamma)$. Moreover, since $\ker \vartheta = \mathrm{Im}\Lambda_{\Gamma}^\#$ and $\ker(d\vartheta + \omega)=\langle E^l\rangle$ are complementary subbundles of $T\Gamma$ and $\alpha$ is a submersion surjective, $\mathfrak{I}(\zeta_0,f_0)=\Lambda_{\Gamma}^{\#}(\alpha^*\zeta_0)+\alpha^*f_0E^l=0$ if and only if $(\zeta_0,f_0)=(0,0)$. So, $\ker \mathfrak{I} = (0,0)$.
\end{proof}

\subsection{Homogeneous twisted symplectic and twisted contact groupoids}
We complete this section by investigating, in the groupoid framework, the relationship which links homogeneous twisted symplectic structures with twisted contact structures. For this reason, we recall the following concepts.

\vspace{2mm}

A \emph{multiplicative vector field on a Lie groupoid} $\Gamma \overset{\alpha}{\underset{\beta}{\rightrightarrows}}\Gamma_0$ \cite{mck} is a pair of vectors fields $(Z,Z_0)$, where $Z\in \Gamma(T\Gamma)$ and $Z_0\in \Gamma(T\Gamma_0)$, such that $Z: \Gamma \to T\Gamma$ is a morphism of Lie groupoids over $Z_0: \Gamma_0 \to T\Gamma_0$, when $T\Gamma \overset{T\alpha}{\underset{T\beta}{\rightrightarrows}}T\Gamma_0$ is endowed with the tangent Lie groupoid structure. The above condition is equivalent to the property that the flows $\Psi_t$ of $Z$ are (local) Lie groupoid automorphisms over the flows $\psi_t$ of $Z_0$.

\vspace{2mm}

A \emph{twisted symplectic groupoid} \cite{cx} is a Lie groupoid $\Gamma \overset{\alpha}{\underset{\beta}{\rightrightarrows}}\Gamma_0$ equipped with a nondegenerate multiplicative $2$-form $\Omega$ on $\Gamma$, i.e., $m^\ast\Omega = pr_1^\ast\Omega + pr_2^\ast\Omega$, where $pr_i : \Gamma_2 \to \Gamma$ is the projection on the $i$-factor, $i=1,2$, of $\Gamma_2$, and a closed $3$-form $\phi_0$ on $\Gamma_0$ such that $d\Omega = \alpha^*\phi_0 - \beta^*\phi_0$. If the closed $3$-form $\phi_0$ is exact, i.e., $\phi_0 = d\omega_0$ with $\omega_0 \in \Gamma(\bigwedge^2T^\ast \Gamma_0)$, we say that $(\Gamma \overset{\alpha}{\underset{\beta}{\rightrightarrows}}\Gamma_0,\Omega, d\omega_0)$ is an \emph{exact twisted symplectic groupoid}. Then, $\Omega$ has the particular type $\Omega = \alpha^*\omega_0 - \beta^*\omega_0 + d\eta$, where $\eta$ is an $1$-form on $\Gamma$, and it is multiplicative if and only if $d\eta$ is.

\vspace{2mm}

Taking into consideration the first footnote, we introduce the notion of a \emph{homogeneous twisted symplectic groupoid} as follows.
\begin{definition}
A \emph{homogeneous twisted symplectic groupoid} is an exact twisted symplectic groupoid $(\Gamma \overset{\alpha}{\underset{\beta}{\rightrightarrows}}\Gamma_0,\Omega, d\omega_0)$ endowed with a multiplicative vector field $(Z,Z_0)$ such that $\mathcal{L}_Z\Omega = \Omega$ and $i(Z_0)d\omega_0 = \omega_0$.
\end{definition}

Hence, we have
\begin{proposition}\label{prop-tcg-tsg}
If $(\Gamma \overset{\alpha}{\underset{\beta}{\rightrightarrows}}\Gamma_0, \vartheta,\omega,r)$, $\omega = \alpha^{\ast}\omega_0-e^{-r}\beta^{\ast}\omega_0$ with $\omega_0\in \Gamma(\bigwedge^2T^{\ast}\Gamma_0)$, is a $2n+1$-dimensional twisted contact groupoid, then its action groupoid $\tilde{\Gamma} \overset{\tilde{\alpha}}{\underset{\tilde{\beta}}{\rightrightarrows}}\tilde{\Gamma}_0$, which is defined by (\ref{act-grpd}), endowed with the structure $(\tilde{\Omega}, d\tilde{\omega}_0, (\frac{\partial}{\partial s},\frac{\partial}{\partial s}))$, where $\tilde{\Omega}=d(e^s\vartheta)+e^s\omega$, $\tilde{\omega}_0 = e^s\omega_0$ and $s$ is the canonical coordinate on $\R$, is a homogeneous twisted symplectic groupoid, and reciprocally.
\end{proposition}
\begin{proof}
By construction, $\dim \tilde{\Gamma} = 2n+2$. So, $\tilde{\Omega}$ is of maximal rank on $\tilde{\Gamma}$, i.e., $\tilde{\Omega}^{2n+2}\neq 0$, if and only if $\vartheta \wedge (d\vartheta + \omega)^n \neq 0$, i.e., $(\vartheta,\omega)$ is a twisted contact structure on $\Gamma$. Furthermore, $d\tilde{\Omega} = d(e^s\omega) = d(e^s(\alpha^{\ast}\omega_0-e^{-r}\beta^{\ast}\omega_0))=d(e^s\alpha^{\ast}\omega_0)-d(e^{s-r}\beta^{\ast}\omega_0) = \tilde{\alpha}^\ast(d\tilde{\omega}_0)-\tilde{\beta}^\ast(d\tilde{\omega}_0)$ and the multiplicativity of $\tilde{\Omega}$ is equivalent with that of $(\vartheta,r)$. Effectively, let $\tilde{pr}_i$ be the projection on the $i$-factor, $i=1,2$, of $\tilde{\Gamma}_2$ and $((g_1, s_2 - r(g_2)),(g_2,s_2))$ an element of $\tilde{\Gamma}_2$. Since $\tilde{\Omega}=e^sds\wedge \vartheta + e^sd\vartheta + e^s\omega$, we have
\begin{equation}\label{eq1-prop}
\tilde{m}^\ast \tilde{\Omega}_{((g_1, s_2 - r(g_2)),(g_2,s_2))} = \tilde{\Omega}_{(m(g_1,g_2),s_2)} = e^{s_2}[ds\wedge \vartheta_{m(g_1,g_2)} + d\vartheta_{m(g_1,g_2)} + \omega_{m(g_1,g_2)}]
\end{equation}
and
\begin{eqnarray}\label{eq2-prop}
(\tilde{pr}_1^\ast\tilde{\Omega} + \tilde{pr}_2^\ast\tilde{\Omega})_{((g_1, s_2 - r(g_2)),(g_2,s_2))} & = & \tilde{\Omega}_{(g_1, s_2 - r(g_2)} + \tilde{\Omega}_{(g_2, s_2)} \nonumber \\
& = & e^{s_2-r(g_2)}d(s-r)\wedge \vartheta_{g_1} + e^{s_2-r(g_2)}(d\vartheta_{g_1} + \omega_{g_1}) \nonumber \\
& & +\,e^{s_2}ds\wedge \vartheta_{g_2} + e^{s_2}d\vartheta_{g_2} + e^{s_2}\omega_{g_2}.
\end{eqnarray}
Hence, $\tilde{m}^\ast \tilde{\Omega}_{((g_1, s_2 - r(g_2)),(g_2,s_2))} = (\tilde{pr}_1^\ast\tilde{\Omega} + \tilde{pr}_2^\ast\tilde{\Omega})_{((g_1, s_2 - r(g_2)),(g_2,s_2))}$ if and only if the components of (\ref{eq1-prop}) and (\ref{eq2-prop}) containing $ds$ are equals and as also ones containing $\omega$. The first equality gives us $\vartheta_{m(g_1,g_2)} = e^{-r(g_2)}\vartheta_{g_1}+\vartheta_{g_2}$, whence we obtain (\ref{mult-cond}). From the second equality, taking into account that $\alpha(g_1 \cdot g_2)=\alpha(g_2)$, $\beta(g_1 \cdot g_2)=\beta(g_1)$ and $\alpha(g_1)=\beta(g_2)$, we get $\omega_{m(g_1,g_2)} = e^{-r(g_2)}\omega_{g_1} + \omega_{g_2} \Leftrightarrow r(g_1\cdot g_2) = r(g_1) + r(g_2)$.

Finally, it is enough to show that $(\frac{\partial}{\partial s},\frac{\partial}{\partial s})$ is a multiplicative vector field on $\tilde{\Gamma} \overset{\tilde{\alpha}}{\underset{\tilde{\beta}}{\rightrightarrows}}\tilde{\Gamma}_0$ and $(\tilde{\Omega}, \tilde{\omega}_0)$ is homogeneous with respect to $(\frac{\partial}{\partial s},\frac{\partial}{\partial s})$. By considering the identifications $T\tilde{\Gamma}\equiv T\Gamma \times T\R$ and $T\tilde{\Gamma}_0\equiv T\Gamma_0 \times T\R$, we can easily check that $T\tilde{\alpha} \circ \frac{\partial}{\partial s} = \frac{\partial}{\partial s} \circ \tilde{\alpha}$, $T\tilde{\beta} \circ \frac{\partial}{\partial s} = \frac{\partial}{\partial s} \circ \tilde{\beta}$ and $T\tilde{m}(\frac{\partial}{\partial s},\frac{\partial}{\partial s})=\frac{\partial}{\partial s} \circ \tilde{m} $. Hence, $(\frac{\partial}{\partial s},\frac{\partial}{\partial s})$ is a multiplicative vector field on $\tilde{\Gamma}$. The homogeneity of $(\tilde{\Omega}, \tilde{\omega}_0)$ with respect to $(\frac{\partial}{\partial s},\frac{\partial}{\partial s})$ is obvious.
\end{proof}

\vspace{2mm}

Moreover, by adapting Theorem 2.6 of A. Cattaneo and P. Xu \cite{cx} in the homogeneous case, we obtain
\begin{theorem}\label{th-ind-hpois}
Let $(\Gamma \overset{\alpha}{\underset{\beta}{\rightrightarrows}}\Gamma_0, \Omega, d\omega_0,(Z,Z_0))$ be a homogeneous twisted symplectic grou\-poid, $(A(\Gamma),[\cdot , \cdot], T\alpha)$ its associated Lie algebroid, and $(\Lambda, d\omega,Z)$, $\omega =\alpha^\ast\omega_0 - \beta^\ast \omega_0$, the homogeneous twisted Poisson structure on $\Gamma$ defined by $(\Omega,d\omega)$. Then
\begin{enumerate}
\item[1)]
$(\Lambda, d\omega,Z)$ induces on $\Gamma_0$ a homogeneous twisted Poisson structure $(\Lambda_0,d\omega_0,Z_0)$ such that $\alpha : \Gamma\to \Gamma_0$ is a twisted Poisson map and $\beta: \Gamma \to \Gamma_0$ is a twisted anti-Poisson map.
\item[2)]
If $(T^*\Gamma_0,\lcf\cdot,\cdot\rcf^{d\omega_0}, \Lambda_0^\#)$ is the Lie algebroid over $\Gamma_0$ determined by $(\Lambda_0,d\omega_0)$, i.e., for any $\zeta_0,\eta_0 \in \Gamma(T^\ast \Gamma_0)$,
\begin{equation*}
\lcf \zeta_0,\eta_0 \rcf^{d\omega_0} = \mathcal{L}_{\Lambda_0^\#(\zeta_0)}\eta_0 - \mathcal{L}_{\Lambda_0^\#(\eta_0)}\zeta_0 - d\Lambda_0(\zeta_0,\eta_0) + d\omega_0 (\Lambda_0^\#(\zeta_0),\Lambda_0^\#(\eta_0),\cdot),
\end{equation*}
the map $\mathfrak{J} : \Gamma(T^*\Gamma_0)\longrightarrow \Gamma(A(\Gamma))$, $\zeta_0 \mapsto \Lambda^{\#}(\alpha^*\zeta_0)$\footnote{Using the same argumentation as in the case of symplectic groupoids \cite{cdw}, we can show that the left invariant vectors fields on a twisted symplectic groupoid $(\Gamma \overset{\alpha}{\underset{\beta}{\rightrightarrows}}\Gamma_0, \Omega, \phi_0)$ is of type $\Lambda^{\#}(\alpha^*\zeta_0)$, where $\zeta_0 \in \Gamma(T^\ast \Gamma_0)$.}, induces a Lie algebroid isomorphism between $(T^*\Gamma_0,\lcf\cdot,\cdot\rcf^{d\omega_0}, \Lambda_0^\#)$ and $(A(\Gamma),[\cdot , \cdot], T\alpha)$.
\end{enumerate}
\end{theorem}
\begin{proof}
Taking into account the results of \cite{cx} and our data $i(Z_0)d\omega_0 = \omega_0$, we have only to check the homogeneity of $\Lambda_0$ with respect to $Z_0$, where $\Lambda_0 =\alpha_\ast \Lambda = -\beta_\ast \Lambda$. Since $\Lambda$ is the inverse of $\Omega$ and $\mathcal{L}_Z\Omega = \Omega$, we get $\mathcal{L}_Z\Lambda = - \Lambda$. Also, from the multiplicativity of $(Z,Z_0)$ on $\Gamma$, we have $T\alpha \circ Z = Z_0\circ \alpha$ and $T\beta \circ Z = Z_0\circ \beta$. So, because $\alpha$ is a submersion surjective, $\mathcal{L}_{Z_0}\Lambda_0 = \mathcal{L}_{\alpha_\ast Z}(\alpha_\ast \Lambda) = \alpha_\ast(\mathcal{L}_Z\Lambda)=\alpha_\ast(-\Lambda)=-\Lambda_0$.
\end{proof}

\vspace{2mm}

Using Proposition \ref{prop-tcg-tsg} and Theorem \ref{th-ind-hpois}, we can formulate the following
\begin{proposition}\label{prop-grpd-basemanif}
Let $(\Gamma \overset{\alpha}{\underset{\beta}{\rightrightarrows}}\Gamma_0, \vartheta,\omega,r)$ be a twisted contact groupoid, $(\Lambda_0,E_0,\omega_0)$ its induced twisted Jacobi structure on $\Gamma_0$ (Theorem \ref{th-ind-jac}) and  $(\tilde{\Gamma} \overset{\tilde{\alpha}}{\underset{\tilde{\beta}}{\rightrightarrows}}\tilde{\Gamma}_0,\tilde{\Omega}, d\tilde{\omega}_0, (\frac{\partial}{\partial s},\frac{\partial}{\partial s}))$ its associated homogeneous twisted symplectic groupoid whose the induced homogeneous twisted Poisson structure on $\tilde{\Gamma}_0$ (Theorem \ref{th-ind-hpois}) is $(\tilde{\Lambda}_0, d\tilde{\omega}_0,\frac{\partial}{\partial s})$. Then, the twisted Poissonization of $(\Lambda_0,E_0,\omega_0)$ coincides with $(\tilde{\Lambda}_0, d\tilde{\omega}_0,\frac{\partial}{\partial s})$.
\end{proposition}

\subsection{Examples of twisted contact and homogeneous twisted symplectic groupoids}
\textbf{1. A twisted contact groupoid constructed by a twisted contact manifold.} Let $(\Gamma_0, \vartheta_0, \omega_0)$ be a $2n+1$-dimensional, simply connected, twisted contact manifold and $(\Lambda_0,E_0)$ the corresponding twisted Jacobi structure on $\Gamma_0$ characterized by (\ref{reeb}) and (\ref{bivect-reeb}). We consider the pair groupoid $\Gamma_0 \times \Gamma_0 \overset{p_2}{\underset{p_1}{\rightrightarrows}} \Gamma_0$ of $\Gamma_0$, $p_i$, $i=1,2$, being the projection on the $i$-factor, and the vector bundle groupoid $\R \overset{\pi}{\underset{\pi}{\rightrightarrows}} \{0\}$. By identifying $\Gamma_0 \times \{0\}$ with $\Gamma_0$, we construct the product groupoid $\Gamma \overset{\alpha}{\underset{\beta}{\rightrightarrows}} \Gamma_0$ of $\Gamma_0\times \Gamma_0$ with $\R$, i.e., $\Gamma = \Gamma_0 \times \Gamma_0 \times \R$. The structure maps $\alpha$ and $\beta$ of $\Gamma$ are, respectively, the projections on the second and first factor of $\Gamma$, that are $\alpha (x,y,t) = y$ and $\beta (x,y,t) = x$. The multiplication map $m$ on $\Gamma_2= \{((x,y,t),(y,z,s)) \, /\,x,y,z \in \Gamma_0 \;\; \mathrm{and}\;\; t,s \in \R \}\subset \Gamma \times \Gamma$ is given by $m((x,y,t),(y,z,s)) = (x,z, t+s)$, the inversion map $\iota : \Gamma \to \Gamma$ by $\iota (x,y,t)=(y,x,-t)$ and the embedding $\varepsilon : \Gamma_0\hookrightarrow \Gamma$ by $\varepsilon(x) = (x,x,0)$. The triple $(\vartheta, \omega, r)$, where
\begin{equation*}
\vartheta = \alpha^\ast \vartheta_0 - e^{-r}\beta^\ast \vartheta_0, \quad \quad \omega = \alpha^\ast \omega_0 - e^{-r}\beta^\ast \omega_0 \quad \quad \mathrm{and} \quad \quad r = p_3,
\end{equation*}
$p_3$ being the projection $\Gamma \to \R$, defines a $r$-multiplicative twisted contact structure on $\Gamma$ whose the Reeb vector field is $E_{\Gamma} = 0 + E_0 + 0$ and its bivector field (\ref{bivect-reeb}) is $\Lambda_{\Gamma} = -e^r\Lambda_0 + \Lambda_0 + 0$. It is easy to verify that
\begin{equation}\label{eq-example}
\vartheta \wedge (d\vartheta +\omega)^{2n+1} = ce^{-(n+1)r} \alpha^\ast(\vartheta_0\wedge (d\vartheta_0 + \omega_0)^n)\wedge \beta^\ast(\vartheta_0\wedge (d\vartheta_0 + \omega_0)^n)\wedge dr \neq 0
\end{equation}
on $\Gamma$, where $c$ is a nonzero constant dependent of $n$, and that $(\vartheta, r)$ satisfies (\ref{mult-cond}). Clearly, the induced twisted Jacobi structure on $\Gamma_0$ by $(\Lambda_{\Gamma},E_{\Gamma},\omega)$ is the initial given one.

\vspace{2mm}
\noindent
\textbf{2. The gauge groupoid of a principal line bundle over a twisted contact manifold.} Let $(\Gamma_0, \vartheta_0, \omega_0)$ be a $2n+1$-dimensional twisted contact manifold, $(\Lambda_0,E_0)$ the corresponding twisted Jacobi structure on $\Gamma_0$ defined by (\ref{reeb}) and (\ref{bivect-reeb}), and $f_0$ a first integral of $E_0$. Let, also, $\pi: P\to \Gamma_0$ be a principal line bundle over $\Gamma_0$, i.e., its structure group $G$ is a $1$-dimensional Lie group that acts freely on $P$ on the right and $P/G = \Gamma_0$, endowed with a connection $\mathfrak{H}$. Precisely, $\mathfrak{H}$ is a $G$-invariant distribution on $P$, complementary to its vertical bundle $\mathfrak{V}=\ker \pi_\ast$. We consider the gauge groupoid $\Gamma \overset{\alpha}{\underset{\beta}{\rightrightarrows}} \Gamma_0$ of $P$ which is the quotient $P\times P /G \rightrightarrows P/G$ of the pair groupoid $P \times P \overset{\pi_2}{\underset{\pi_1}{\rightrightarrows}} P$ of $P$, $\pi_i$, $i=1,2$, being the projection on the $i$-copy of $P$, by the right diagonal action of $G$ on $P\times P$. The natural projection $\varpi : P\times P \to P\times P /G$, $(p_1,p_2)\overset{\varpi}{\mapsto}[(p_1,p_2)]=\{(p_1g,p_2g)\,/\, g\in G\}$, is a Lie groupoid morphism over $\pi$. Thus, the source map $\alpha$ and the target map $\beta$ of $\Gamma$ are, respectively, the maps defined by $\alpha \circ \varpi = \pi \circ \pi_2$ and $\beta \circ \varpi= \pi \circ \pi_1$. Hence, for each $[(p_1,p_2)]\in \Gamma$, $\alpha([(p_1,p_2)])=\pi(p_2)$ and $\beta([(p_1,p_2)])=\pi(p_1)$, while the inversion $\iota : \Gamma \to \Gamma$ and the embedding $\varepsilon : \Gamma_0 \hookrightarrow \Gamma$ are given, respectively, by $\iota ([(p_1,p_2)]) = [(p_2,p_1)]$ and $\varepsilon(x) = [(p,p)]$, where $p\in \pi^{-1}(x)$. Moreover, the product of two composable elements $[(p_1,p_2)]$, $[(q_1,q_2)]$ of $\Gamma$, that means that $\alpha([(p_1,p_2)])=\pi(p_2)=\pi(q_1)=\beta([(q_1,q_2)])$, is $m([(p_1,p_2)],[(q_1,q_2)])=[(p_1,q_2g)]$, where $g$ is the unique element of $G$ for which $p_2 = q_1g$. The triple $(\vartheta,\omega,r)$, where
\begin{equation*}
\vartheta = \alpha^\ast \vartheta_0 - e^{-r}\beta^\ast \vartheta_0, \quad \quad \omega = \alpha^\ast \omega_0 - e^{-r}\beta^\ast \omega_0 \quad \quad \mathrm{and} \quad \quad r=\alpha^\ast f_0 - \beta^\ast f_0,
\end{equation*}
defines a $r$-multiplicative twisted contact structure on $\Gamma$, i.e., (\ref{eq-example}) and (\ref{mult-cond}) hold on $\Gamma$ and $\Gamma_2$, respectively. In order to present the tensors fields $\Lambda_{\Gamma}$ and $E_{\Gamma}$ that define the corresponding $(d\omega,\omega)$-twisted Jacobi structure on $\Gamma$, we recall that given a $q$-vector field $Q_0$ on $\Gamma_0$ there exists a unique $q$-vector field $Q_0^h$ on $P$, called the \emph{horizontal lift of $Q_0$ on $P$ with respect to $\mathfrak{H}$}, such that, for any $p\in P$, $Q_0^h(p)\in \bigwedge^q \mathfrak{H}_p$ and $Q_0^h$ is $\pi$-related to $Q_0$, i.e.,
\begin{equation*}
Q_0^h (\pi^\ast \zeta_1,\ldots,\pi^\ast \zeta_q) = Q_0(\zeta_1,\ldots,\zeta_q)\circ \pi, \quad \quad \mathrm{for}\;\mathrm{all}\; \zeta_1,\ldots,\zeta_q \in \Gamma(T^\ast \Gamma_0).
\end{equation*}
Because of the $G$-invariance of $\mathfrak{H}$, any $q$-vector field of type $Q_0^h + Q_0'^h$ on $P\times P$ descends to the quotient $P\times P /G$. Hence, we may easily verify that
\begin{equation*}
\Lambda_{\Gamma} = \varpi_\ast (-e^{\varpi^\ast r}\Lambda_0^h + \Lambda_0^h) \quad \quad \mathrm{and} \quad \quad E_{\Gamma}=\varpi_\ast(0+E_0^h).
\end{equation*}
Obviously, the induced twisted Jacobi structure on $\Gamma_0$ by $(\Lambda_{\Gamma},E_{\Gamma},\omega)$ is the initial given one.

\vspace{2mm}
\noindent
\textbf{3. An homogeneous twisted symplectic groupoid.} Let $(S,\omega)$ be a symplectic manifold, $\Pi = \Pi^\#(\omega)$ the corresponding nondegenerate Poisson tensor and $(\Gamma_0, t\Pi)$, $\Gamma_0=S\times \R$ and $t$ being the canonical coordinate on $\R$, the associated Heisenberg-Poisson manifold \cite{w-bull} to $S$. We suppose that $(\Gamma_0, t\Pi)$ is integrable \cite{w-bull} and we denote by $(\Gamma,\Omega)\overset{\alpha}{\underset{\beta}{\rightrightarrows}}\Gamma_0$ the symplectic groupoid which integrates it. Since $t\Pi$ is an exact Poisson structure \cite{duf-zung} (there is the vector field $T_0= -t\frac{1}{t}$ which satisfies $t\Pi = [t\Pi, T_0]$), the multiplicative symplectic form $\Omega$ on $\Gamma$ is also exact \cite{cr-fer-Luxem}. Precisely, the flow of $T_0$ integrates to the flow of a vector field $T$ on $\Gamma$ such that $\Omega=d(i(T)\Omega)$ \cite{cr-fer-Luxem}. So, $\Omega=d\sigma$, where $\sigma = i(T)\Omega$, and it is homogeneous with respect to $T$.

Now, we consider on $\Gamma_0$ a twisted Poisson structure $(\Lambda, \varphi)$ such that $\Lambda$ and $\varphi$ are projectable along the integral curves of $T_0$. (For example, we can take $\Lambda = f\Pi$ and $\varphi = -f^{-2}df\wedge \omega$ with $f$ a nonconstant function on $S$.) Let $\Omega_0$ be the corresponding nondegenerate twisted symplectic $2$-form on $T^\ast \Gamma_0$ \cite{royt, ca} which, in a local coordinate system $(x_1,\ldots,x_n,p_1,\ldots,p_n)$, $n=\dim \Gamma_0$, of $T^\ast \Gamma_0$, is written as
\begin{equation*}
\Omega_0 = \sum_{i=1}^n dp_i \wedge dx_i + \frac{1}{2}\sum_{i,j,k,l=1}^n p_i\lambda^{ij}\varphi_{jkl}dx_k\wedge dx_l,
\end{equation*}
where $\lambda^{ij}$ and $\varphi_{jkl}$ are, respectively, the local components of $\Lambda$ and $\varphi$. We remark that $\Omega_0$ is homogeneous with respect to the Liouville vector field $Z_0 = \sum_{i=1}^{n}p_i \frac{\partial}{\partial p_i}$ on $T^\ast \Gamma_0$.

On the other hand, we consider a left Lie groupoid action of $\Gamma$ on the canonical projection $q: T^\ast \Gamma_0 \to \Gamma_0$ with action map $\Phi : \Gamma \star T^\ast \Gamma_0 \to T^\ast \Gamma_0$, where $\Gamma \star T^\ast \Gamma_0 = \{(g,z)\in \Gamma \times T^\ast \Gamma_0 \,/ \, \alpha(g) = q(z)\}$, compatible with the structure maps of $\Gamma$, \cite{duf-zung}. Let $\Gamma\ltimes T^\ast \Gamma_0 \overset{\tilde{\alpha}}{\underset{\tilde{\beta}}{\rightrightarrows}}T^\ast\Gamma_0$ be the semi-direct product groupoid of $\Gamma$ with $T^\ast \Gamma_0$, \cite{duf-zung}. We have $\Gamma\ltimes T^\ast \Gamma_0 = \Gamma \star T^\ast \Gamma_0$, $\tilde{\alpha}(g,z)=z$, $\tilde{\beta}(g,z) = \Phi(g,z)$, $\tilde{\iota}(g,z)=(g^{-1},\Phi(g,z))$, $\tilde{\varepsilon}(z) = (\varepsilon(q(z)),z)$ and $\tilde{m}\big((g,z),(g',z')\big)=(m(g,g'),z')$ with $z=\Phi(g',z')$. The multiplicative nondegenerate $2$-form
\begin{equation*}
\tilde{\Omega} = \tilde{\alpha}^\ast \omega_0 - \tilde{\beta}^\ast \omega_0 + d\tilde{\eta},
\end{equation*}
where
\begin{equation*}
\omega_0 = \frac{1}{2}\sum_{i,j,k,l=1}^n p_i\lambda^{ij}\varphi_{jkl}dx_k\wedge dx_l \quad  \mathrm{and}  \quad \tilde{\eta} = \tilde{\alpha}^\ast(\sum_{i=1}^n p_idx_i)-\tilde{\beta}^\ast(\sum_{i=1}^n p_idx_i) + \sigma,
\end{equation*}
endows $\Gamma\ltimes T^\ast \Gamma_0$ with a $(\tilde{\alpha}^\ast d\omega_0 - \tilde{\beta}^\ast d\omega_0)$-twisted symplectic structure which is homogeneous with respect to the multiplicative vector field $(T+(T_0^\ast+Z_0), T_0^\ast+Z_0)$\footnote{We recall that the space tangent to $\Gamma\ltimes T^\ast \Gamma_0$ at a point $(g,z)$ is the vector subspace of $T_{(g,z)}(\Gamma \times T^\ast \Gamma_0)=T_g\Gamma \times T_zT^\ast \Gamma_0$, $\{(X,Y)\in T_g\Gamma \times T_zT^\ast \Gamma_0 \,/\, \alpha_\ast X(\alpha(g))=q_\ast Y(q(z))\}$ (see \cite{lm}, p. 345). Hence, we have that $T_g + (T_0^\ast + Z_0)_z$ is tangent to $\Gamma\ltimes T^\ast \Gamma_0$ at the point $(g,z)$ because $\alpha_\ast T (\alpha(g))=T_{0 \alpha(g)}=q_\ast(T_0^\ast+Z_0)(q(z))$. Also, we have $T\tilde{\alpha} \circ(T+(T_0^\ast+Z_0))(g,z)=(T_0^\ast+Z_0)_z =(T_0^\ast+Z_0)\circ \tilde{\alpha}(g,z)$ and $T\tilde{\beta} \circ(T+(T_0^\ast+Z_0))(g,z)=\Phi_\ast(T+(T_0^\ast+Z_0))(\Phi(g,z))=(T_0^\ast+Z_0)_{\Phi(g,z)}=(T_0^\ast+Z_0)\circ \tilde{\beta}(g,z)$, whence we obtain the multiplicativity of $(T+(T_0^\ast+Z_0), T_0^\ast+Z_0)$. The seconde equation of the above sequence is deduced from the compatibility condition $q\circ \Phi = \beta \circ p_1$ of $\Phi$ with $\beta$, where $p_1$ is the projection of $\Gamma\ltimes T^\ast \Gamma_0$ on $\Gamma$.} on $\Gamma\ltimes T^\ast \Gamma_0\rightrightarrows T^\ast \Gamma_0$. In above, $T_0^\ast$ denotes the vector field on $T^\ast \Gamma_0$ which, in the local coordinate system $(x_1,\ldots,x_n,p_1,\ldots,p_n)$ of $T^\ast \Gamma_0$, is written as $T_0^\ast = T_0 + \sum_{i = 1}^{n}0\frac{\partial}{\partial p_i}$.

\section{Integration of twisted Jacobi structures}\label{sect-integr}
In this section we study the inverse problem of that which was presented in Theorem \ref{th-ind-jac}. Specifically, we discuss the following question: \emph{Given a twisted Jacobi manifold $(M,\Lambda,E,\omega)$, does there exist a twisted contact groupoid, with manifold of units $M$, such that its associated Lie algebroid is isomorphic to $(T^\ast M\times \R,\{\cdot,\cdot\}^{\omega}, \pi \circ (\Lambda,E)^\#)$?} When this is the case, we say that $(M,\Lambda,E,\omega)$ is \emph{integrable}.

The above problem is a special case of the problem of integration of a Lie algebroid which asks for the existence of a Lie groupoid whose Lie algebroid is isomorphic to a given one. It was a longstanding problem of Differential Geometry which was solved, in 2003, by M. Crainic and R.L. Fernandes \cite{cr-fer-Lie}. However, the preceded works of J. Pradines \cite{pr} and R. Almeida with P. Molino \cite{almo} for Lie algebroids integrated by local Lie groupoids, K. Mackenzie \cite{mck1} for transitive Lie algebroids, P. Dazord \cite{dz1} and P. Dazord with G. Hector \cite{dzh} for the Lie algebroids structures on the cotangent bundles of regular Poisson manifolds, and A.S. Cattaneo with G. Felder \cite{cf} for the Lie algebroids structures on the cotangent bundles of arbitrary Poisson manifolds, had also an important contribution in its study.

In specific, M. Crainic and R.L. Fernandes \cite{cr-fer-Lie} have proved that with each given Lie algebroid $(A, [\cdot,\cdot],\rho)$ over a smooth manifold $M$, we can associate a topological groupoid $G(A)$, called the \emph{Weinstein groupoid of $A$}, with simply connected $\beta$-fibres\footnote{We defined the Lie algebroid $A(\Gamma)$ of a Lie groupoid $\Gamma \overset{\alpha}{\underset{\beta}{\rightrightarrows}}\Gamma_0$ as $A(\Gamma)=\ker \beta_\ast \cap T_{\Gamma_0}\Gamma$, while M. Crainic and R.L. Fernandes in \cite{cr-fer-Lie} identified $A(\Gamma)$ with $\ker \alpha_\ast \cap T_{\Gamma_0}\Gamma$. For this reason, we must construct $G(A)$ by changing the roles of $\alpha$ and $\beta$ in the construction of $G(A)$ presented in \cite{cr-fer-Lie}. On the other hand, we know that the inversion map of a Lie groupoid sends $\alpha$-fibers to $\beta$-fibers, and reciprocally. So, if the $\alpha$-fibers are simply connected, then the $\beta$-fibers are also simply connected, and reciprocally.}, which is the quotient of the space of $A$-paths in $A$ by an homotopy equivalence relation, and they have obtained necessary and sufficient conditions under which $G(A)$ is a Lie groupoid. Whenever this holds, the Lie algebroid of $G(A)$ is isomorphic to $A$.

\vspace{2mm}

Next, we will describe, in brief, the construction of the Weinstein groupoid $G(A)$ for a given Lie algebroid $(A, [\cdot,\cdot],\rho)$, and some properties of $G(A)$ acquired from the presence on $A$ of an $1$-cocycle of its Lie algebroid cohomology complex with trivial coefficients.

\subsection{The Weinstein groupoid}
\begin{definition}[\cite{cr-fer-Lie}]
Let $\pi: A\to M$ be a Lie algebroid and $\rho : A\to TM$ its anchor map. An \emph{$A$-path} of $A$ is a smooth map $c : I=[0,1]\to A$ of class $C^1$ such that, for all $t\in I$,
\begin{equation*}
\rho(c(t)) = \frac{d\gamma(t)}{dt},
\end{equation*}
where $\gamma(t)=\pi\circ c (t)$ is a path in $M$ and it is called the \emph{base path} of $c$. The set of all $A$-paths of $A$ is denoted by $P(A)$.
\end{definition}

\begin{definition}[\cite{cr-fer-Lie}]
An $A$-\emph{connection} on a vector bundle $E$ over $M$ is a bilinear map $\nabla : \Gamma(A)\times \Gamma(E) \to \Gamma(A)$, $(\zeta,\xi)\to \nabla_{\zeta}\xi$, such that, for any $f\in \C$ and $(\zeta,\xi)\in \Gamma(A)\times \Gamma(E)$,
\begin{equation*}
\nabla_{f\zeta}\xi = f\nabla_{\zeta}\xi  \quad \quad \mathrm{and} \quad \quad \nabla_{\zeta}(f\xi) = f\nabla_{\zeta}\xi+\mathcal{L}_{\rho(\zeta)}(f)\xi.
\end{equation*}
If $E=A$, we define the \emph{torsion} of the $A$-connection $\nabla$ on $A$ as the $\C$-bilinear map $T_{\nabla} : \Gamma(A)\times \Gamma(A) \to \Gamma(A)$ given, for any $\zeta_1,\zeta_2\in \Gamma(A)$, by
\begin{equation*}
T_{\nabla}(\zeta_1,\zeta_2) = \nabla_{\zeta_1}\zeta_2 - \nabla_{\zeta_2}\zeta_1 - [\zeta_1,\zeta_2].
\end{equation*}
\end{definition}

We note that with each standard $TM$-connection $\nabla$ on $A$, we can associate an obvious $A$-connection on $A$, denoted also by $\nabla$, by setting, for any $\zeta_1,\zeta_2\in \Gamma(A)$, $\nabla_{\zeta_1}\zeta_2 \equiv \nabla_{\rho(\zeta_1)}\zeta_2$. Let $\nabla$ be a such connection on $A$ and $\partial_t$ the induced derivative operator which associates to each $A$-path $c: I\to A$ the path $\partial_t c$ in $A$ which is the $\nabla$-horizontal component of $\frac{dc}{dt}$, i.e., if $\gamma : I\to M$ is the base path of $c$ and $\zeta$ is a time dependent section of $A$ such that $\zeta(t,\gamma(t))=c(t)$, then
\begin{equation*}
\partial_t c = \nabla_{\rho(c(t))}\zeta + \frac{d\zeta}{dt}.
\end{equation*}

\begin{definition}[\cite{cr-fer-Lie}]
An $A$-\emph{homotopy} is a family $c_{\epsilon}:I\to A$, $c_{\epsilon}(t)=c(\epsilon,t)$, of $A$-paths which depends on a parameter $\epsilon \in I$ in a $C^2$-fashion and which has the following properties: (i) their base paths $\gamma_{\epsilon}: I\to M$, $\gamma_{\epsilon}(t)=\gamma(\epsilon,t)$, have fixed ends points and (ii) the solution $b(\epsilon,t)$ of the equation
\begin{equation*}
\partial_tb - \partial_{\epsilon}c = T_{\nabla}(c,b),  \quad \quad b(\epsilon,0)=0,
\end{equation*}
satisfies, for all $\epsilon\in I$, $b(\epsilon,1)=0$.\footnote{We recall that the solution $b$ does not depend on $\nabla$, \cite{cr-fer-Lie}.}

Two $A$-paths $c_0$ and $c_1$ are said to be \emph{homotopic}, and we write $c_0 \sim c_1$, if there exists an $A$-homotopy $c: I\times I \to A$ joining them, i.e., $c(0,t)=c_0(t)$ and $c(1,t)=c_1(t)$.

For a given $A$-path $c$, we denote by $[c]$ the set of all homotopic $A$-paths to $c$.
\end{definition}

We define now the Weinstein groupoid $G(A)\overset{\alpha}{\underset{\beta}{\rightrightarrows}}M$ of $\pi: A\to M$ as the space of classes of homotopic paths of $P(A)$, i.e.,
\begin{equation*}
G(A) : = P(A)/\sim \,,
\end{equation*}
with the following structure maps. For any $[c]\in G(A)$, the source map $\alpha : G(A)\to M$ and the target map $\beta :G(A)\to M$ are given, respectively, by $\alpha([c])= \gamma(0)$ and $\beta([c])=\gamma(1)$, where $\gamma$ is the base path of $c$, while the inversion $\iota : G(A)\to G(A)$ maps $[c]$ to the class $[\bar{c}]$ of the opposite $A$-path $\bar{c}$ of $c$, i.e., for any $t\in [0,1]$, $\bar{c}(t) = -c(1-t)$, and the embedding $\varepsilon : M\hookrightarrow G(A)$ maps each point $x$ of $M$ to the class $[0_x]$ of the constant trivial path above $x$. The set $G_2(A)$ of composable pairs of $G(A)\times G(A)$ consists of the pairs of classes $([c_1],[c_0])$ whose representatives $(c_1,c_0)$ are pairs of composable $A$-paths, i.e., $\pi(c_1(0))=\pi(c_0(1))$, and the product map $m$ on $G_2(A)$ is given by $m([c_1],[c_0])=[c_1]\cdot[c_0] = [c_1\odot c_0]$, where $c_1\odot c_0$ is the concatenation of $(c_0,c_1)$\footnote{The path $c_1\odot c_0$ is only piecewise smooth. But, by choosing an appropriate cutoff function $\tau\in C^\infty(\R)$, we reparameterize $c_i$, $i=0,1$, by setting $c_i^{\tau}(t):= \tau'(t)c_i(\tau(t))$. Then, $c_1^{\tau}\odot c_0^{\tau}$ is a smooth $A$-path, homotopic to $c_1\odot c_0$. For details, see \cite{cr-fer-Lie}.} defined as follows
\begin{equation*}
c_1\odot c_0 (t)\equiv \left\{
\begin{array}{lc}
2c_0(2t),& 0\leq t\leq \frac{1}{2}, \\
\\

2c_1(2t-1),& \frac{1}{2}< t\leq 1.
\end{array}
\right.
\end{equation*}

\begin{theorem}[\cite{cr-fer-Lie}]\label{th-integ-Lie alg}
The groupoid $G(A)$ is a $\beta$-simply connected topological groupoid. Moreover, whenever $A$ is integrable, $G(A)$ admits a smooth structure which makes it into the unique $\beta$-simply connected Lie groupoid integrating $A$.
\end{theorem}

As M. Crainic and R.L. Fernandes have shown \cite{cr-fer-Lie}, $G(A)$ can be viewed as the leaf space of a foliation $\mathcal{F}(A)$ on $P(A)$, of finite codimension, whose leaves are just the classes of homotopic $A$-paths. In particular, $\mathcal{F}(A)$ may be defined by the orbits of a Lie algebra action on the space $\tilde{P}(A)$ of all $C^1$ curves $c: I \to A$ with base path $\gamma= \pi\circ c$ of class $C^2$, which contains $P(T^\ast M)$ as submanifold. An alternative description of $G(A)$, when $A$ is the Lie algebroid $T^\ast M$ of a twisted Poisson manifold $M$, is developed by A.S. Cattaneo and P. Xu \cite{cx}. By modifying the method introduced in \cite{cf} for Poisson manifolds, they have obtained $G(T^\ast M)$ by an appropriate symplectic reduction of the cotangent bundle $T^\ast P(M)$ of the space $P(M)$ of all paths on $M$ of class $C^2$. In fact, the above two descriptions of $G(T^\ast M)$ coincid since they use the same Lie algebra action and since $\tilde{P}(T^\ast M)$ is identified with $T^\ast P(M)$. For details, we can consult \cite{cr-fer-Poisson} and \cite{ca}.

\subsection{Effects of an $1$-cocycle of $A$ on $G(A)$}
Let $R$ be an $1$-cocycle in the Lie algebroid cohomology complex with trivial coefficients of $(A,[\cdot,\cdot],\rho)$, i.e., $R$ is a section of the dual vector bundle $A^\ast\to M$ of $A \to M$ such that, for any $\zeta,\eta \in \Gamma(A)$,
\begin{equation*}
\langle [\zeta,\eta], R\rangle = \mathcal{L}_{\rho(\zeta)}\langle \eta,R \rangle - \mathcal{L}_{\rho(\eta)}\langle \zeta,R \rangle,
\end{equation*}
and $\mathbf{ac}^R : \Gamma(A) \to \Gamma(T(M\times \R))$ the map given, for any $\zeta \in \Gamma(A)$, by
\begin{equation*}
\mathbf{ac}^R(\zeta) = \rho(\zeta)+\langle \zeta, R\rangle\frac{\partial}{\partial s},
\end{equation*}
$s$ being the canonical coordinate on the factor $\R$ of $M\times \R$. It is easy to show that $\mathbf{ac}^R$ is an action of $A$ on the fibered manifold $\varpi : \tilde{M}=M\times \R \to M$ in the sense of \cite{hmck}. Thus, the vector bundle $\tilde{A}=A\times \R \to M\times \R$, which is isomorphic \cite{im} to the pull-back bundle $\varpi^\ast A \to M\times \R$ of $A$ over $\varpi$, admits \cite{hmck} a Lie algebroid structure $([\cdot,\cdot]^R, \rho^R)$. Having identified $\Gamma(\tilde{A})$ with the set of the time-dependent sections of $A$, the bracket $[\cdot,\cdot]^R$ and the anchor map $\rho^R$ are defined, for any $\tilde{\zeta}, \tilde{\eta} \in \Gamma(\tilde{A})$, respectively, by
\begin{equation*}
[\tilde{\zeta},\tilde{\eta}]^{R} =
[\tilde{\zeta},\tilde{\eta}] + \langle\tilde{\zeta},R\rangle\frac{\partial\tilde{\eta}}{\partial s} - \langle\tilde{\eta},R\rangle\frac{\partial\tilde{\zeta}}{\partial s}
\hspace{3mm}\mathrm{and} \hspace{3mm} \rho^{R}(\tilde{\zeta}) = \rho(\tilde{\zeta})
+\langle\tilde{\zeta},R\rangle\frac{\partial}{\partial s}.
\end{equation*}

\vspace{1mm}

We denote by $G(\tilde{A})$ the Weinstein groupoid of $\tilde{A}$. We will examine the relationship which link $G(\tilde{A})$ with $G(A)$.

\vspace{2mm}

Any section $R$ of $A^\ast$ can be integrating over an $A$-path $c: I \to A$ by setting
\begin{equation}\label{def-integr}
\int_c R : = \int_0^1 \langle c(t), R(\pi(c(t)))\rangle dt.
\end{equation}
A basic property of the integral (\ref{def-integr}) is its invariance under an $A$-homotopy, if $R$ is an $1$-cocycle of $A$. It is proved in \cite{cr-fer-Poisson} for the $1$-cocycles of Lie algebroids coming from Poisson manifolds, i.e., for the Poisson vectors fields, but the same method of proof may be used for the $1$-cocycles of any Lie algebroid. Therefore, if $R$ is an $1$-cocycle of $A$ and $c_0$, $c_1$ are two homotopic $A$-paths, we have
\begin{equation*}
\int_{c_0}R = \int_{c_1}R.
\end{equation*}
From the above property we deduce that the map $r : P(A) \to \R$, $c \to \int_c R$, descends to a well defined map, also denoted by $r$, on the quotient space $G(A)$ by setting
\begin{equation*}
r([c]) : = \int_c R.
\end{equation*}
The additivity of integration with respect the concatenation of paths shows that $r$ is a multiplicative function on $G(A)$, i.e., for any $([c_1],[c_0])\in G_2(A)$,
\begin{equation*}
r([c_1]\cdot[c_0])= r([c_1])+r([c_0]).
\end{equation*}
Hence, we can construct the action groupoid $\widetilde{G(A)} \overset{\tilde{\alpha}}{\underset{\tilde{\beta}}{\rightrightarrows}}\tilde{M}$ associated to the left action $\mathbf{ac}^r$ of $G(A)$ on $\varpi : \tilde{M}= M\times \R \to M$ given by (\ref{act-grpd}). We have
\begin{proposition}[\cite{cr-zhu}]\label{prop-isomor}
(i) The topological groupoids $\widetilde{G(A)}$ and $G(\tilde{A})$ are isomorphic. (ii) $A$ is integrable if and only if $\tilde{A}$ is. In this case, the previous isomorphism is a Lie groupoid isomorphism.
\end{proposition}

\subsection{Integration of twisted Jacobi structures}
Let $(M,\Lambda,E,\omega)$ be a twisted Jacobi manifold, $\big(T^*M\times \R,\{\cdot,\cdot\}^{\omega}, \pi \circ (\Lambda,E)^\#, (-E,0)\big)$ its associated Lie algebroid with $1$-cocycle and $G(T^*M\times \R)$ the Weinstein groupoid of $T^*M\times \R$. Let, also, $(\tilde{M},\tilde{\Lambda},\tilde{\omega})$ be the twisted Poissonization of $(M,\Lambda,E,\omega)$, i.e., $\tilde{M}=M\times \R$, $\tilde{\Lambda} = e^{-s}(\Lambda + \frac{\partial}{\partial s}\wedge E)$ and $\tilde{\omega} =e^s\omega$, $(\lcf \cdot,\cdot\rcf^{d\tilde{\omega}},\tilde{\Lambda}^\#)$ the corresponding Lie algebroid structure on $T^\ast \tilde{M}$, and $G(T^\ast \tilde{M})$ the Weinstein groupoid of $T^\ast \tilde{M}$. By applying the results of the previous subsection to the case where $(A,R)$ is the Lie algebroid with $1$-cocycle $(T^*M\times \R, (-E,0))$ and taking into account that $(T^*M\times \R)\times \R\cong T^*(M\times \R)$, we get, as direct consequence of Proposition \ref{prop-isomor}, the following
\begin{proposition}\label{prop-is-jac-hpois}
(i) There is an isomorphism of topological groupoids
\begin{equation}\label{isom-grpd}
\widetilde{G(T^\ast M \times \R)}\cong G(T^\ast \tilde{M}).
\end{equation}
(ii) The Lie algebroid $T^*M\times \R$ is integrable if and only if $T^\ast \tilde{M}$ is.
\end{proposition}

We now establish, via Proposition \ref{prop-tcg-tsg}, the inverse of Theorem \ref{th-ind-jac}.
\begin{theorem}\label{integr-tj}
An integrable twisted Jacobi manifold $(M,\Lambda,E,\omega)$ is integrated by a twisted contact groupoid.
\end{theorem}
\begin{proof}
The assumption $(M,\Lambda,E,\omega)$ is integrable means that the corresponding Lie algebroid $(T^*M\times \R,\{\cdot,\cdot\}^{\omega}, \pi \circ (\Lambda,E)^\#)$ is integrable. Consequently, according to Proposition \ref{prop-is-jac-hpois}, $(T^\ast \tilde{M},\lcf \cdot,\cdot\rcf^{d\tilde{\omega}},\tilde{\Lambda}^\#)$ is also integrable; fact which implies that $(\tilde{M},\tilde{\Lambda},\tilde{\omega},\frac{\partial}{\partial s})$ is integrable as homogeneous twisted Poisson manifold. Crainic's and Fernandes's Theorem show that $T^\ast M \times \R$ (resp. $T^\ast \tilde{M}$) is integrated by $G(T^\ast M \times \R)\overset{\alpha}{\underset{\beta}{\rightrightarrows}}M$ (resp. $G(T^\ast \tilde{M})$) which is the unique target-simply connected Lie groupoid integrating it. Furthermore, Cattaneo's and Xu's Theorem in \cite{cx} assures us that $G(T^\ast \tilde{M})$ is endowed with a non-degenerate, multiplicative, twisted symplectic form $\tilde{\Omega}$. Because of (\ref{isom-grpd}), it induces on $\widetilde{G(T^\ast M \times \R)}\overset{\tilde{\alpha}}{\underset{\tilde{\beta}}{\rightrightarrows}}\tilde{M}$ a non-degenerate, multiplicative, $(\tilde{\alpha}^\ast d\tilde{\omega} - \tilde{\beta}^\ast d\tilde{\omega})$-twisted symplectic form denoted, also, by $\tilde{\Omega}$. From construction, $(\frac{\partial}{\partial s},\frac{\partial}{\partial s})$ is a multiplicative vector field on $\widetilde{G(T^\ast M \times \R)} = G(T^\ast M \times \R) \times \R$ whose the corresponding vector field $\frac{\partial}{\partial s}$ on the base manifold $\tilde{M}=M\times \R$ is the homothety vector field of the $d\tilde{\omega}$-twisted Poisson structure $\tilde{\Lambda}$. We need to show that $\tilde{\Omega}$ is homogeneous with respect to $\frac{\partial}{\partial s}$.

The multiplicativity of $\frac{\partial}{\partial s}$ on $\widetilde{G(T^\ast M \times \R)}=G(T^\ast M \times \R) \times \R$ implies that its flows
\begin{eqnarray*}
\Psi_u : G(T^\ast M \times \R) \times \R & \longrightarrow & G(T^\ast M \times \R) \times \R \\
(g, s) & \longrightarrow & (g, s+u)
\end{eqnarray*}
are (local) Lie groupoid automorphisms over the transformations $\psi_u : M\times \R \to M \times \R$ defined by $\psi_u (x,s)=(x,s+u)$. Thus, $\tilde{\alpha}\circ \Psi_u^{-1} = \psi _u^{-1} \circ \tilde{\alpha}$. Let $\tilde{\Lambda}_G$ be the $(\tilde{\alpha}^\ast d\tilde{\omega} - \tilde{\beta}^\ast d \tilde{\omega})$-twisted Poisson structure on $\widetilde{G(T^\ast M \times \R)}$ defined by the inversion of $\tilde{\Omega}$. From Theorem 2.6 of \cite{cx}, we have that $\tilde{\Lambda}_G$ is projectable on $\tilde{M}$ via $\tilde{\alpha}_\ast$ and its projection is the $d\tilde{\omega}$-twisted Poisson structure $\tilde{\Lambda}$ which is homogeneous with respect to $\frac{\partial}{\partial s}$. On the other hand, we have $\psi _u^\ast \tilde{\Lambda} = (\psi _u^{-1})_\ast \tilde{\Lambda} = e^{-u}\tilde{\Lambda}$. Consequently, $(\tilde{\alpha}_\ast \circ (\Psi_u^{-1})_\ast)\tilde{\Lambda}_G = ((\psi _u^{-1})_\ast \circ \tilde{\alpha}_\ast)\tilde{\Lambda}_G \Leftrightarrow \tilde{\alpha}_\ast ((\Psi_u^{-1})_\ast \tilde{\Lambda}_G)=e^{-u}\tilde{\Lambda}$, whence, taking into account the expression of $\Psi_u$, we obtain $(\Psi_u^{-1})_\ast \tilde{\Lambda}_G = e^{-u}\tilde{\Lambda}_G$. Hence, $\mathcal{L}_{\frac{\partial}{\partial s}}\tilde{\Lambda}_G = \frac{d}{du}(\Psi_u^\ast \tilde{\Lambda}_G)\vert_{u=0} = \frac{d}{du}(e^{-u}\tilde{\Lambda}_G)\vert_{u=0} = -\tilde{\Lambda}_G$, which means that $\tilde{\Lambda}_G$ is homogeneous with respect to $\frac{\partial}{\partial s}$. For this, its inverse $\tilde{\Omega}$ is also homogeneous with respect to $\frac{\partial}{\partial s}$, i.e., $\mathcal{L}_{\frac{\partial}{\partial s}}\tilde{\Omega} = \tilde{\Omega}$.

We remark that $\tilde{\Omega}$, as an exact twisted symplectic form, can be written as $\tilde{\Omega} = \tilde{\alpha}^\ast \tilde{\omega} - \tilde{\beta}^\ast \tilde{\omega} + d\tilde{\eta}$, where $\tilde{\eta}$ is an $1$-form on $\widetilde{G(T^\ast M \times \R)}$ such that $d\tilde{\eta}$ is multiplicative. Since the part $\tilde{\alpha}^\ast \tilde{\omega} - \tilde{\beta}^\ast \tilde{\omega}$ of $\tilde{\Omega}$ is homogeneous with respect to $\frac{\partial}{\partial s}$, it is clear that the homogeneity of $\tilde{\Omega}$ implies that of $d\tilde{\eta}$, i.e., $\mathcal{L}_{\frac{\partial}{\partial s}}d\tilde{\eta} = d\tilde{\eta} \Leftrightarrow d(i(\frac{\partial}{\partial s})d\tilde{\eta}) = d\tilde{\eta}$. We set $\tilde{\vartheta} = i(\frac{\partial}{\partial s})d\tilde{\eta}$ and we check that $\mathcal{L}_{\frac{\partial}{\partial s}}\tilde{\vartheta} = \tilde{\vartheta}$. So, we have $\tilde{\vartheta} = e^s \vartheta$, where $\vartheta$ is a $1$-form on $G(T^\ast M \times \R)$, and $d\tilde{\vartheta} = d\tilde{\eta}$. Therefore, $\tilde{\Omega} = \tilde{\alpha}^\ast \tilde{\omega} - \tilde{\beta}^\ast \tilde{\omega} + d\tilde{\vartheta} = e^s(\alpha^\ast \omega - e^{-r}\beta^\ast \omega) + d(e^s \vartheta)$. According to Proposition \ref{prop-tcg-tsg}, we conclude that $(\vartheta, \alpha^\ast \omega - e^{-r}\beta^\ast \omega)$ defines a twisted contact structure on $G(T^\ast M \times \R)$ whose induced twisted Jacobi structure on the manifolds of units is the initial given one (see, Proposition \ref{prop-grpd-basemanif})\end{proof}

\end{document}